\documentclass[12pt,reqno,oneside]{amsart}
\makeindex
\usepackage{amsmath,amsthm,amssymb}
\usepackage{graphicx,xcolor}
\usepackage{float} %allow to insert images where they should be

\usepackage{etoolbox}
\numberwithin{figure}{section}
\numberwithin{equation}{section}
 
\DeclareFontFamily{U}{mathb}{\hyphenchar\font45}
\DeclareFontShape{U}{mathb}{m}{n}{
	<-6> mathb5 <6-7> mathb6 <7-8> mathb7
	<8-9> mathb8 <9-10> mathb9
	<10-12> mathb10 <12-> mathb12
}{}
\DeclareSymbolFont{mathb}{U}{mathb}{m}{n}
\DeclareMathSymbol{\llcurly}{\mathrel}{mathb}{"CE}
\DeclareMathSymbol{\ggcurly}{\mathrel}{mathb}{"CF}
\definecolor{Red}{cmyk}{0,1,1,0}

\definecolor{Blue}{cmyk}{1,1,0,0}

\theoremstyle{plain}
\newtheorem{maintheorem}{Theorem} 

\newtheorem{maincorollary}[maintheorem]{Corollary}
\newtheorem{theorem}{Theorem }[section]
\newtheorem{proposition}[theorem]{Proposition}
\newtheorem{lemma}[theorem]{Lemma}
\newtheorem{corollary}[theorem]{Corollary}

\theoremstyle{definition} \theoremstyle{remark}
\newtheorem{remark}[theorem]{Remark}
\newtheorem{example}[theorem]{Example}
\newtheorem{definition}[theorem]{Definition}

\newcommand{\vep}{\varepsilon}

\newcommand{\diam}{\operatorname{diam}}
\newcommand{\dist}{\operatorname{dist}}

\newcommand{\interior}{\operatorname{int}}

\renewcommand{\ge}{\geqslant}
\renewcommand{\geq}{\geqslant}
\renewcommand{\leq}{\leqslant}
\renewcommand{\le}{\leqslant}

\usepackage{mathrsfs}
\usepackage[colorlinks=true,linkcolor=blue, urlcolor=red, citecolor=blue]{hyperref}	% Links in the pdf. Backref option: each 

\begin{document}
\pagestyle{myheadings}

\title[Stability Theory for Local Iterated Function Systems]{Stability Theory for Local Iterated Function Systems}

\author{Elismar R. Oliveira and Paulo Varandas}

\address{Elismar R. Oliveira,  Departamento de Matem\'atica, Universidade Federal do Rio grande do Sul\\
Porto Alegre, Brazil}
\email{elismar.oliveira@ufrgs.br}

\address{Paulo Varandas, Departamento de Matem\'atica, Universidade Federal da Bahia\\
Av. Ademar de Barros s/n, 40170-110 Salvador, Brazil \& Center for Research and Development in Mathematics and Applications (CIDMA), Department of Mathematics, University of Aveiro, 
3810--193 Aveiro, Portugal
}
\email{paulo.varandas@ua.pt}

\date{\today}

\begin{abstract}
We develop a stability theory for contractive local IFSs on compact metric spaces. Unlike the classical global setting, local systems may exhibit a richer symbolic and geometric structure, including code spaces that are not of finite type and attractors with endpoints, leading to new mechanisms of instability.
We first prove that concordant shadowing implies upper semicontinuity of the local attractor and persistence of the code space, yielding a criterion for combinatorial stability under perturbations. Under the open set condition, we establish a strong form of topological stability for combinatorially stable contractive local systems, and prove the converse implication on compact manifolds of dimension at least three.
In particular, we show that contractive graph-directed IFSs are topologically stable. We also construct contractive local IFSs derived from $\beta$-transformations that are combinatorially unstable. These results show that stability in the local setting is governed by the interplay between contraction and the combinatorial rigidity of the code space. Applications to graph-directed IFSs and pseudogroup actions are also given.
\color{black}
\end{abstract}
%\tableofcontents

\maketitle

\section{Introduction}
\label{sec:intro}

The theory of Iterated Function Systems (IFSs) is a central component of modern fractal geometry and dynamical systems. In its classical form, an IFS consists of a compact metric space $X$ together with a finite family of continuous maps $(f_i)_{1 \le i \le n}$ acting on $X$. The associated Hutchinson--Barnsley operator $F:2^X \to 2^X$,
\[
F(B)=\bigcup_{1\le j \le n} f_j(B),
\]
is continuous with respect to the Hausdorff metric and, under uniform contractivity, admits a unique non-empty compact invariant set $\Lambda \subset X$ satisfying $F(\Lambda)=\Lambda$. In this setting, $\Lambda$ attracts every compact subset of $X$, and the dynamics is encoded by the full shift on $\{1,\dots,n\}^{\mathbb{N}}$ (see e.g. \cite{Bar06,Hut81}). Under the open set condition (OSC), each point of $\Lambda$ admits a unique coding by an itinerary in $\{1,\dots,n\}^{-\mathbb{N}}$. Moreover, classical IFSs have rich dynamical and geometric features, very well studied in the past decades, and the attractor of the IFS possesses a fractal structure whose Hausdorff dimension can be computed as a zero of a certain pressure function (see e.g. \cite{Fal94,Fal03,MU96}). 

Local Iterated Function Systems (local IFSs), introduced by Barnsley~\cite{Bar86}, replace global maps by maps $f_j:X_j \to X$ defined on closed subsets $X_j \subset X$. Writing $R_{\mathfrak X}=(X_j,f_j)_{1\le j\le n}$, one associates the operator
\[
F_{\mathfrak X}(B)=\bigcup_{1\le j \le n} f_j(B \cap X_j), \quad B\in K^*(X).
\]
In contrast with the classical case, $F_{\mathfrak X}$ may fail to be continuous or contractive, compositions of maps may not be globally admissible, and symbolic descriptions may break down. These features, already noted in early work (see, e.g., \cite{Fis95}), reflect the intrinsically local nature of the system and the absence of a general fixed-point principle.

\smallskip
A systematic framework for local IFSs was recently initiated in \cite{OV25a}. Even in the contractive case, the dynamics departs significantly from the classical theory: $F_{\mathfrak X}$ need not be contracting, invariant sets need not be unique, and attractors may contain points with finite orbits (endpoints) \cite{OV25a}. Moreover, the symbolic dynamics is governed by admissible compositions and typically corresponds to a proper invariant subset of the full shift, not necessarily of finite type. In particular, local IFSs may produce attractors that are not self-similar.
These phenomena reveal a richer combinatorial and geometric structure, and call for a corresponding notion of stability. In the classical setting, hyperbolicity entails semicontinuous dependence of the attractor on the defining maps and yields topological stability through shadowing-type properties. In fact the notion of concordant shadowing, which may be viewed as an analogue of shadowing for non-autonomous dynamical systems, has been shown to hold for both contractive IFSs,  surjective expansive IFSs and open classes of hyperbolic IFSs (cf. \cite{AT25,CRV19,GV06}). These results support the principle that uniform hyperbolicity provides a sufficient condition for the concordant shadowing property.
Extending this theory to local IFSs faces two fundamental obstacles: the symbolic space may change discontinuously under perturbations, and small changes in the maps or domains may alter admissibility, create or destroy endpoints, and reorganize the attractor.

\smallskip
In this paper we develop a stability theory for contractive local IFSs on compact metric spaces. Our approach builds on the orbit and coding structures introduced in \cite{OV25a} and addresses the interaction between admissibility, symbolic dynamics, and geometry.
Our first main result shows that concordant shadowing implies upper semicontinuity of the local attractor (Theorem~\ref{thm:shadowing-usc}) and that its robust occurrence yields persistence of the symbolic structure, leading to a criterion for combinatorial stability (see Corollary~\ref{thm:robust-concordantshadowing}). 
Upper semicontinuity of the local attractor holds, in general, for contractive local IFSs when viewed as a function of the domains of the maps (see Theorem~\ref{thm:usc} for a precise statement).
Under the OSC, we prove that combinatorial stability implies a strong form of topological stability, and we establish a converse on compact manifolds of dimension at least three. In particular, within this class, concordant shadowing and topological stability are equivalent.
These results exhibit a new stability mechanism, governed not only by hyperbolicity but also by the rigidity of the admissible symbolic structure. 
At the level of the code space of local attractors, this is related to a classical theorem of Walters \cite{Wa78}, which characterizes subshifts of finite type as the class of subshifts that satisfy the shadowing property (see also Good and Meddaugh~\cite{GM20} for a generalization to more general compact spaces). 
In fact, in contrast with the context of IFSs, contractivity alone does not guarantee concordant shadowing, and even systems satisfying the OSC may fail to exhibit it (see Theorem~\ref{thm:beta}). In particular, the principle that hyperbolicity is sufficient to ensure the concordant shadowing property fails in the setting of local IFSs (see e.g. Example~\ref{ex:shift2}).

In this direction, we provide a sharp characterization of stability phenomena for contractive local IFSs under the open set condition (Theorems ~\ref{thm:combinatorial-topological} and \ref{thm:concordantshadowing-implies-stable}). On the one hand, Theorem ~\ref{thm:combinatorial-topological} shows that combinatorial stability, together with contraction, suffices to guarantee topological stability, thus extending classical hyperbolic principles to the local setting in the presence of symbolic rigidity. On the other hand, Theorem \ref{thm:concordantshadowing-implies-stable} establishes a partial converse in the context of manifolds, proving that topological stability together with combinatorial stability forces the concordant shadowing property. In particular, when combined, these results yield an equivalence between topological stability and concordant shadowing within this class, revealing that stability in local IFSs is not solely dictated by contraction, but rather by a delicate balance between geometric hyperbolicity and the persistence of the admissible symbolic structure. As an immediate consequence that illustrate the scope of our results, the former provides a complete characterization of topological stability within the class of combinatorially stable contractive local IFSs on compact manifolds of dimension larger or equal than three (Corollary
~\ref{cor:stable-iff-concordantshadowing}), 
shows that contractive graph-directed IFSs inherit both combinatorial and topological stability (Corollary~\ref{cor:graphdirected}), and yields applications to the context of pseudogroup actions (see Subsection~\ref{subsec:pseudo}). 

The paper is organized as follows. In Section~\ref{sec:mainresults} we introduce the general setting of local IFSs and state our main results. In Section~\ref{sec:prelim} we collect and prove preliminary results, including the description of the code space and the space of orbits, and basic notions of stability and shadowing for local IFSs. Section~\ref{sec:usc} is devoted to the analysis of concordant shadowing and its consequences for the upper semicontinuity of the local attractor, as well as its relation with combinatorial stability.
Section~\ref{sec:shadows} is devoted to  quantitative analysis of the concordant shadowing property, with refined key estimates on shadowing points and pseudo-orbits, and characterization of admissible pseudo-orbits. 
In Section~\ref{sec:local-beta} we construct first examples of contractive local IFSs, derived from $\beta$-transformations, which fail to satisfy the concordant shadowing property and exhibit combinatorial and topological instability. Section~\ref{sec:contractive-stable} is devoted to the discussion of the topological stability for contractive local IFSs and its relation with combinatorial stability. In particular we proved that contractive IFSs are topologically stable, a fact that we could not find in the literature. 
Finally, in Section~\ref{sec:examples} we present further examples illustrating both stability and instability phenomena, including systems with rich symbolic behavior, and discuss applications to pseudogroup actions.

\color{black}

\section{Main results}\label{sec:mainresults}

This section is devoted to the statement of our main results. First we recall the general setting, and introduce the notions of concordant shadowing, combinatorial and topological stability. 

\subsection{Local IFSs}
The concept of local iterated function systems on arbitrary topological spaces which can be traced back to \cite{BH93,Fis95}. 
Let $(X,d)$ be a compact metric space and let $ K^*({X})$ denote the hyperspace of compact subsets of $X$.  

\begin{definition}\label{def:local IFS}
   A \emph{local IFS} 
   \index{Local IFS} 
   on $X$ is a structure $R_{\mathfrak X}=(X_j, f_{j})_{1\le j \le n}$ where 
  $n \geq 2$ is an integer and, for each $1\le j \le n$, one has that $X_j \in K^*({X})$ is a closed subset of $X$, and $f_{j}: X_{j} \to X$ is a continuous map. We shall refer to the family $\mathfrak X=(X_j)_{1\le j \le n}$ as the set of \emph{restrictions}\index{Restriction! For a local IFS} 
  of the local IFS.
  Moreover, we say that the local IFS
    $R_{\mathfrak X}$ is \emph{contractive} \index{Contractive local IFS} if there exists $\lambda_j\in [0,1)$ so that 
    $ 
    d(f_j(x), f_j(y))\le \lambda_j d(x,y) 
    $ 
    for every $x,y\in X_j.$
\end{definition}

Given  $R_{\mathfrak X}=(X_j, f_{j})_{1\le j \le n}$, $k\ge 1$ and $\underline a =(a_0,a_1,\ldots) \in \{1,2,,\dots, n\}^{\mathbb N}$ we denote the compositions of maps in  $R_{\mathfrak X}$ by 
\[f_{\underline a}^k:=f_{a_{k-1}}\circ f_{a_{k-2}}\circ \cdots\circ f_{a_{0}},\] 
whenever the latter makes sense.

   The \emph{local Hutchinson-Barnsley operator}\index{Local Hutchinson-Barnsley operator} $F_{\mathfrak X}:  K^*({X}) \to K^*(X) $ is defined as
\begin{equation}
    \label{def:loc fract operat}
    F_{\mathfrak X}(B)=\bigcup_{1\le j \le n} f_{j}(B \cap X_{j}),
\end{equation}
for each $B \in K^*(X)$. 
While the previous operator need not be a contraction, there exists a unique compact $F_{\mathfrak X}$-invariant  subset  $A_{\mathfrak X}\subseteq X$ which is maximal with respect to the inclusion and, if non-empty, it is referred to as the \emph{local attractor} for the local IFS  $R_{\mathfrak X}$ (cf. \cite[Theorem~A]{OV25a}). Local attractors may contain endpoints and points with infinite orbits. In this way, it is also natural to consider the following natural subsets 
\begin{equation}
    \label{eq:inclusionsA}
A^\infty_{\mathfrak X}
    \subseteq 
    A_{\mathfrak X}
    \subseteq 
    A_R
\end{equation}
where $A_R$ denotes the attractor for the IFS $R=(X,f_i)_{1\le i \le n}$ and
\begin{equation}
    \label{eq:maxinvA}
    A^{\infty}_{\mathfrak X} =\Big\{x\in A_{\mathfrak X} \colon \exists (a_j)_{j\ge 0} \text{ so that }  (f_{a_j}\circ \dots \circ f_{a_0})(x) \in \bigcup_{i=1}^n X_i, \, \text{for every}\, j\ge 0\Big\}.
\end{equation}
stands for the set of points in the local attractor which have an infinite orbit. Moreover, while the code space of the local attractor $A_{\mathfrak X}$ may be a proper invariant subset of $\Sigma^-=\{1,2\dots, n\}^{-\mathbb N}$, in the classical framework of IFSs the code space is the full shift $\Sigma^-$ and the attractor is the maximal invariant set, meaning that one has equalities in \eqref{eq:inclusionsA}. 
We refer the reader to Section~\ref{sec:prelim} for more details.

%%%%%%%%%%%%%%%%%%%%%%%%%%%%%%%%%%%%%%%%%%%
\subsection{Concordant shadowing  and applications}

In this subsection we adapt the notions of transitivity and shadowing for local IFSs, which will allow one to compare approximate trajectories with genuine orbits and to detect robust dynamical behavior under perturbations.

\begin{definition}\label{def:transitive}
Let $R_{\mathfrak X}=(X_j,f_j)_{1\le j\le n}$ be a local IFS.  
We say that the local attractor $A_{\mathfrak X}$ is \emph{transitive} if for any pair of nonempty open sets $U,V\subset A_{\mathfrak X}$ there exist $k\ge 1$ and indices $a_1,\dots,a_k\in\{1,\dots,n\}$ such that
\[
    (f_{a_k}\circ\cdots\circ f_{a_1})(U)\cap V\neq\emptyset.
\]
\end{definition}

\begin{definition}\label{def:pseudo-orbit}
Let $R_{\mathfrak X}=(X_j,f_j)_{1\le j\le n}$ be a local IFS. Given $\underline{a} \in \Sigma^+$ and $\delta>0$, a sequence $(x_k)_{k\ge 0}$ in $X$ is called a \emph{$(\underline{a},\delta)$-pseudo-orbit} if $x_k \in X_{a_{k}}$ and $d\bigl(f_{a_k}(x_k),\,x_{k+1}\bigr)<\delta$ for each $k\ge 0$. 
\end{definition}
It is evident from the definition that if $A_{\mathfrak X}$ denotes the local attractor and if the points $x_k$ in the pseudo-orbit belong to $A_{\mathfrak X}$  then $f_{a_k}(x_k) \in A_{\mathfrak X} $, for all $k \geq 0$.

\begin{definition}
\label{def:shadowing}
A local IFS $R_{\mathfrak X}$ satisfies the \emph{shadowing property} if for every $\vep>0$ there exists $\delta>0$ such that for every $( {\underline{a}},\delta)$-pseudo-orbit $(x_k)$ in $A_{\mathfrak X}$ there exist a point $x\in A_{\mathfrak X}$ and a sequence $\tilde {\underline{a}}\in \Sigma^+$ for which
\[
    d\bigl(f_{\tilde {\underline{a}}}^k(x),\,x_k\bigr)<\vep
    \quad\text{for all } k\ge 1.
\]
In this case we say that the $( {\underline{a}},\delta)$-pseudo-orbit $(x_k)$ is   \emph{$(\tilde {\underline{a}},\vep)$-shadowed} by $x$.
\end{definition}

As defined above, while the shadowing property allows to construct genuine orbits from approximate trajectories it does not allow for such a process to be obtained under the same compositions of maps that generated the approximate pseudo-orbit. This motivates the following notion.

\begin{definition}
\label{def:concordant-shadowing}
We say that a local IFS $R_{\mathfrak X}$ satisfies the \emph{concordant shadowing property} if for every $\vep>0$ there exists $\delta>0$ such that every $(\underline{a},\delta)$-pseudo-orbit $(x_k)$ in $A_{\mathfrak X}$ is $(\underline{a},\vep)$-shadowed by some $x\in A_{\mathfrak X}$. 
\end{definition}

The concordant shadowing property appears as a natural condition in the stability of the local attractors, namely to address the dependence of the local attractor
$A_{\mathfrak X}$ with respect to the maps and their domains. 
First, let us recall that, given a compact metric space $(X,d)$, a set function $F: K^*(X) \to K^*(X)$ is \emph{upper semicontinuous}  \index{Upper semicontinuous} at $B\in K^*(X)$ if for every open set $V \subseteq X$ such that
$ 
F(B) \subseteq V,
$
there exists an open neighborhood $U$ of $B$ in $K^*(X)$  such that
\[
F(A) \subseteq V \; \text{for all }  A\in U.
\]
In other words, for every open set $V  \subset X$ containing $F(B)$ there exists $\delta>0$ such that $F(B') \subset V$ for every $B' \in K^*(X)$ so that $dist_H(B',B)<\delta$. 
Our first main result ensures that the concordant shadowing property guarantees the
upper-semicontinuity of the local attractor. More precisely:

\begin{maintheorem}
    \label{thm:shadowing-usc}
    Let $(X,d)$ be a compact metric space and assume that the local contractive IFS $R_{\mathfrak X}= \left(X_j, f_j \right)_{1\le j \le n}$ satisfies the  concordant shadowing property. Then there exists an open neighborhood $\mathcal V$ of
    $R_{\mathfrak X}$ such that  
    $\Sigma_{\mathfrak Y}^-\subset \Sigma_{\mathfrak X}^-$
    for every $R_{\mathfrak Y} \in \mathcal V$.
    In particular the local attractor map
$$\mathcal V\ni R_{\mathfrak Y} \mapsto A_{\mathfrak Y}
$$
is upper-semicontinuous at $R_{\mathfrak X}$.
\end{maintheorem}

It is important to emphasize that the shadowing property is not necessary for the conclusion of Theorem~\ref{thm:shadowing-usc}. Indeed, upper-semicontinuity holds trivially in the degenerate case where all maps $f_j$ coincide with the identity, although this family does not satisfy the concordant shadowing property. 
More generally, in the absence of perturbations, upper-semicontinuity of the attractor remains valid without assuming concordant shadowing, as established by the following result.

\begin{maintheorem}
	\label{thm:usc}
	Assume that $(X,d)$ is a compact metric space and let $f_j : X \to X$ be a contraction for each $1\le j \le n$. Then, the local attractor map 
	$$(K^*(X))^n \ni \mathfrak  X    \mapsto A_{\mathfrak X}$$ 
	is upper-semicontinuous (in the Vietoris topology).
\end{maintheorem}

\begin{remark}
	The map 
	$ 
	(K^*(X))^n \ni \mathfrak X \mapsto A_{\mathfrak X}$ 
	need not be lower-semicontinuous in general (see e.g. Example~\ref{examp: non-semicont}).
\end{remark}
\color{black}

\medskip
Given a contractive local IFS  $R_{\mathfrak X}=(X_j,f_j)_{1\le j \le n}$ we proceed to characterize the both the space of pseudo-orbits that are shadowable and the space of points whose orbits appear in the shadowing property. Consider the spaces
\begin{equation}
    \label{eq:def-enlarged-Omega0}
    \Omega =\Big\{((x_i,a_i))_{i\ge 0} \in (X\times \{1, 2, \dots, n\})^{\mathbb N} \colon x_{i+1}=f_{\underline a_i}(x_i) \; \text{for every}\; i\ge 0 \Big\}.
\end{equation}
and 
\begin{equation}
    \label{eq:def.Ia}
    I(\underline{a}) = \pi_{0,X} \left( \left\{ (x_i, a_i)_{i \geq 0} \in \Omega \mid \pi_{\Sigma^+} \left( (x_i, a_i)_{i \geq 0} \right) = \underline{a} \right\} \right) \subset X_{a_0}.
\end{equation}
Due to the finiteness and continuity of the maps $f_j$, it follows that $I(\underline{a}) \in K^*(X_{a_0})$. Moreover,
it is clear that if $x\in X$ is an $(\underline a,\vep)$-shadow for some $(\underline{a}, \delta)$-pseudo-orbit can only occur if
$$
\underline{a} \in \pi_{\Sigma^+}(\Omega) =: \Sigma_{\infty}
\quad\text{and}\quad
x\in I(\underline a).
$$
This suggests to define, 
 	for $\underline{a} \in \Sigma^+$, the set $\Gamma_{\underline a, \delta}$ as the set of all $(\underline{a}, \delta)$-pseudo-orbits.
    Notice that $\Gamma_{\underline a, \delta} \supset \Gamma_{\underline a, \delta'}$ if $\delta > \delta'$. We note that if $\underline{a}\in \Sigma_{\infty}$ if $\bigcap_{\delta>0} \Gamma_{\underline a, \delta} \neq \emptyset$
	(we say $\underline a$ is \textit{significant}).
	\bigskip

	\begin{maintheorem}\label{thm:Y} Let $R_{\mathfrak X}$ be a contractive local IFS.
    The following are equivalent:
    \begin{enumerate}
        \item[(B1)] $R_{\mathfrak{X}}$ has the concordant shadowing property;
        \item[(B2)]
        for any significant sequence 
        $\underline{a} \in \Sigma_{\infty}$, 
		\[ 
        \lim_{\delta \to 0} \operatorname{dist}_H(\pi_0(\Gamma_{\underline a, \delta}), I(\underline{a})) = 0. 
        \]
    \end{enumerate}
    In consequence, combinatorial stability is equivalent to robustness of condition (B2).
	\end{maintheorem}

The proof of Theorem~\ref{thm:Y} relies on fine estimates for the distances of shadowing points and pseudo-orbits, which may be of independent interest 
(cf. Lemma~\ref{lem:L1} and  Proposition~\ref{prop:X} for more details). 

\color{black}

    \begin{remark}
        In seeking a criterion for shadowing in local IFSs, the Open Set Condition (OSC) is of paramount importance. Without it, one can construct local IFSs with $(\underline{a}, \delta)$-pseudo-orbits for arbitrarily small $\delta$ that cannot be $(\underline{a}, \varepsilon)$-shadowed as in 
        Example~\ref{ex:shift2}
    \end{remark}

%%%%%%%%%%%%%%%%%%%%%%%%%%%%%%%%%%%%%%%%%%%
\subsection{Negative shadowing property}

In what follows we will introduce a concept of negative shadowing for local IFSs with the OSC. In rough terms, we aim at a shadowing property for points in the attractor such that remain close user negative iterations. 
Let us make it precise.  
By the open set condition, for every point $x\in A_{\mathfrak X}$ there exists a unique code $\underline b\in\Sigma^-_{\mathfrak X}$ such that
$x=\pi_{\mathfrak X}(\underline b)$. Hence the code map
$\pi_{\mathfrak X}:\Sigma^-_{\mathfrak X}\rightarrow A_{\mathfrak X}$
is an homeomorphism between compact metric spaces and, in particular, both $\pi_{\mathfrak X}$ and
$\pi_{\mathfrak X}^{-1}$ are uniformly continuous.
Define the map $T:A_{\mathfrak X}\to A_{\mathfrak X}$ by
\begin{equation}
    \label{eq:defT}
T(x)=\pi_{\mathfrak X}(\sigma^-(\underline b)) \in f_{b_{-1}}^{-1}(\{x\}),
\; \text{where } \underline b=\pi_{\mathfrak X}^{-1}(x).
\end{equation}
which defines a special negative iterate of the local IFS.
Given $\delta>0$, we say that a sequence $(x_k)_{k\le -1}$ in the local attractor $A_{\mathfrak X}$ is a \emph{negative $\delta$-pseudo orbit} if 
$$
d\Big(T(x_{-k}), \; x_{-k-1}\Big)<\delta
\quad \text{for every $k\le -1$.}
$$
\begin{definition}\label{def:negative-concordant-shadowing}
We say that $R_{\mathfrak X}$ satisfies the \emph{negative shadowing property} if for every $\vep>0$ there exists $\delta>0$ such that every negative $\delta$-pseudo-orbit $(x_{k})_{k\le -1}$ in $A_{\mathfrak X}$ there exists $x\in A_{\mathfrak X}$ so that 
\begin{equation}
    \label{eq:shadowinversex}
    d\Big( T^k(x), \; x_{-k} \Big)<\vep,
\end{equation}
in which case we say that $(x_{k})_{k\le -1}$ is 
\emph{$\vep$-shadowed} by $x\in A_{\mathfrak X}$.  
\end{definition}

We observe that, by construction (recall ~\eqref{eq:defT}), one has that 
$$
\pi_{\mathfrak X}(\sigma^-(\underline b))
= T(x)=T(\pi_{\mathfrak X}(\underline b))
$$
for every $\underline b\in\Sigma^-_{\mathfrak X}$, where
$x=\pi_{\mathfrak X}(\underline b)$.
This shows that 
$ 
\pi_{\mathfrak X}\circ \sigma^-\mid_{\Sigma^-_{\mathfrak X}} = T\circ \pi_{\mathfrak X}$,
hence $\sigma^-\mid_{\Sigma^-_{\mathfrak X}}$ and $T$ are topologically conjugate. 
In consequence, $R_{\mathfrak X}=(X_j,f_j)_{1\le j \le n}$ satisfies the negative shadowing property if and only if the same holds for $T$ and $\sigma^-\mid_{\Sigma_{\mathfrak X}^-}.$ 
In particular, using  
Walters' classification
\cite{Wa78} of the shift-invariant sets that satisfy the shadowing property we deduce 
the following immediate consequence.

\begin{maincorollary}
 \label{thm:negativeshadowing-classification}
Let $R_{\mathfrak X}= \left(X_j, f_j \right)_{1\le j \le n}$ be  a contractive local IFS satisfying the OSC. Then
$R_{\mathfrak X}$ satisfies the negative concordant shadowing property if and only if the code space $\Sigma^-_{\mathfrak X}$ is a subshift of finite type.
\end{maincorollary}

\color{black}

\subsection{Combinatorial stability}

Local IFSs may exhibit a rich and intricate combinatorial structure. In  particular, the code space of a contractive local IFS $R_{\mathfrak X}=(X_j,f_j)_{1\le j \le n}$ may be a proper shift-invariant subset of the full symbolic space $\Sigma^-$, and it may change under perturbations of the domains $X_j$ and the maps $f_j$ within an appropriate topology, that we now define. 
Consider the spaces of local IFSs is defined by
    $$
\mathcal F_C=\Big\{R_{\mathfrak X}=(X_j,f_j)_{1\le j \le N} \,\colon\,    \mathfrak X=(X_j)_{1\le j \le N} \in (K^*(X))^N\; \text{and}\; f_j: X_j \to X \; \text{is a contraction}\Big\}
$$
and
$$
\mathcal F_0=\Big\{R_{\mathfrak X}=(X_j,f_j)_{1\le j \le N} \,\colon\,    \mathfrak X=(X_j)_{1\le j \le N} \in (K^*(X))^N\; \text{and}\; f_j: X_j \to X \; \text{is continuous}\Big\}. 
$$
\color{black}
We also define 
$\mathcal F_*^{(n)}=\Big\{  R_{\mathfrak X} \in \mathcal F_* \colon N=n\Big\}$ for $*\in\{C,0\}$. 
In order to avoid superfluous information, when comparing two local IFSs $R_{\mathfrak X}, R_{\mathfrak y} \in \mathcal F_0$ and studying its topological stability we will require them to have the same number of domains, and for the domains and maps to be close (in an appropriate sense to be defined below). 

\medskip
Let us define a Skorohod-type distance in the space of local IFSs.
In order to do so, we consider local IFSs parameterized by the number $n$ of domains.  
 For each pair $X_i, Y_i \in K^*(X)$, the set
\[
\mathcal{H}(X_i, Y_i) :=\{\zeta: X_i \to Y_i \colon  \text{$\zeta$ is an homeomorphism}\}
\]
allows to pair and parametrize homeomorphic domains of local IFSs.
Now, fix $L>0$ such that $L> 3 \operatorname{diam}(X)$ and define
\begin{equation}
    \label{eq:distanceS}
    d^0_{\mathcal S}(f_i,g_i)
    :=\begin{cases}
    {\inf}_{\zeta \in \mathcal{H}(X_i, Y_i)} \; \{ {d_{C^0}(\, \zeta, id_{X_i}\,)+} \;d_{C^0}(\, f_i, g_i\circ \zeta\,)\} & \text{ if } \mathcal{H}(X_i, Y_i) \neq \emptyset\\
    L  & \text{ otherwise} 
\end{cases},
\end{equation}
{were $d_{C^0}(\, \zeta, id_{X_i}\,)$ stands for the $C^0$-distance given by $\sup_{x \in X_i} d(\zeta(x), x)$.}
Now, given two local IFSs $R_{\mathfrak X}=(X_j,f_j)_{1\le j \le n}, R_{\mathfrak Y}=(Y_j,f_j)_{1\le j \le n}$ we define their $D$-distance as
\begin{equation}
\label{def:distanceS}
D\Big(R_{\mathfrak X}, R_{\mathfrak Y}\Big)
    = \max_{1\le i \le n} \text{dist}_H(X_i,Y_i) + \max_{1\le i \le n} d_{\mathcal S}(f_i,g_i)     
\end{equation}
where the Skorohod-type distance $d_{\mathcal S}$ 
is defined by 
$d_{\mathcal S}(f_i,g_i) =\max\{d^0_{\mathcal S}(f_i,g_i),d^0_{\mathcal S}(g_i,f_i)\}$
(we refer the reader e.g. to \cite{Witt} for the definition of Skorohod metrics on cadlag spaces).
It is not hard to show that $D$ defines a metric.

\begin{remark}
We note that classical IFSs are combinatorially stable (the code space of arbitrary perturbations is the full shift) and that an easier topology on the space of IFSs is obtained by identification with  elements in $C(\{1,..,n \}\times X , X)$  making the situation substantially simpler
(see \cite{AT25}).
\end{remark}
 
We are now in a position to formulate the notion of combinatorial stability for local IFSs.

\begin{definition}\label{def:combinatorially stable}
We say that a local IFS  $R_{\mathfrak X}=(X_i,f_i)_{1\le i \le n}$ is \emph{combinatorially stable}\index{Combinatorially stable} if there exists an open neighborhood $\mathcal V$ of $R_\mathfrak X$  so that 
$ \Sigma^-_{\mathfrak X} = \Sigma^-_{\mathfrak Z}$ for every $R_\mathfrak Z\in \mathcal V$. 
\end{definition}

By definition, the notion of combinatorial stability means that the code spaces of local attractors for nearby local IFSs remain the same, hence do not go through any  bifurcation. In what follows we will provide a criterion for combinatorial stability.

\begin{definition}
    We say that $R_{\mathfrak X}$ satisfies the \emph{
    robust concordant shadowing property} 
    if for every $\vep>0$ there exists an open neighborhood $\mathcal V$ 
        of $R_\mathfrak X$ and $\delta>0$ so that for every $R_\mathfrak Y\in \mathcal V$, every $(\underline a,\delta)$-pseudo orbit for $R_{\mathfrak Y}$  
    is $(\underline a,\vep)$-shadowed by an orbit of $R_{\mathfrak Y}$.
\end{definition}

As a byproduct of Theorem~\ref{thm:shadowing-usc} we obtain that the robust concordant shadowing property suffices for combinatorial stability.

\begin{maincorollary}
\label{thm:robust-concordantshadowing}
Let $(X,d)$ be a compact metric space. If $R_{\mathfrak X}= \left(X_j, f_j \right)_{1\le j \le n}$ satisfies the 
robust concordant shadowing property then 
$R_{\mathfrak X}$ is combinatorially  stable.
\end{maincorollary}

It might be expected that any local IFS $R_{\mathfrak X}$ satisfying the open set condition and hyperbolicity, under additional assumptions that the local attractor $A_{\mathfrak X}$ as the denseness of periodic orbits, would be combinatorially stable. However, this is not the case.

\begin{maintheorem}
\label{thm:beta}
There exists a contractive local IFS $R_{\mathfrak X}=(X_j,f_j)_{1\le j \le n}$ such that: 
\begin{itemize}
    \item[(i)]  $R_{\mathfrak X}$ does not satisfy the concordant shadowing property;
    \item[(ii)] $R_{\mathfrak X}$ is not combinatorially stable;
    \item[(iii)] $R_{\mathfrak X}$ satisfies the open set condition;
    \item[(iv)] the set of periodic orbits is dense in the code space $\Sigma_{\mathfrak X}^-$.
\end{itemize}
\end{maintheorem}

We highlight that, in the proof of previous theorem, we construct first examples of combinatorially unstable local IFSs. Such examples are derived from $\beta$-transformations, for certain special parameters $\beta$.

\begin{remark}
    \label{rmk:graph-directed-combin-stable}
    On the positive direction, Proposition~\ref{thm:criterion-combinatorial-stability} offers a criterion, combining the open set condition and a Markov type relation, to ensure combinatorial stability of local IFSs. This is applicable to graph-directed IFSs, hence the latter are combinatorially stable. 
\end{remark}

 %%%%%%%%%%%%%%%%%%%%%%%%%%%%%%%%%%%%%%%%%%%
\subsection{Topological stability}

Although the notion of combinatorial stability allows one to distinguish local IFSs coded by shift-invariant subsets that are not topologically conjugate, it is easy to construct examples of local IFSs sharing the same code space while exhibiting substantially different orbit structures on their attractors. This motivates a notion of stability formulated in terms of the existence of topological conjugacies on the space of orbits.
In order to formulate thus notion we recall that the dynamics of points in the local attractor are described by a bilateral shift-invariant subset $\widehat{\Sigma}_{\mathfrak X}\subset \Sigma^-_{\mathfrak X}\times \Sigma^+$  and by its projection onto the positive coordinates, which we denote by $\widehat{\Sigma}_{\mathfrak X}^+$ (cf. Theorem~\ref{thmC}). 

\begin{definition}
\label{def:conjugated}
    Let $X$ be a compact metric space. We say the local IFSs $R_{\mathfrak X}=(X_j, f_{j})_{1\le j \le n}$ and $R_{\mathfrak Y}=(Y_j, g_{j})_{1\le j \le n}$ are \emph{topologically conjugate}\index{Topologically conjugate} if there exists an homeomorphism $\tau: \widehat \Sigma^+_{\mathfrak X} \to \widehat \Sigma^+_{\mathfrak Y}$ and for every $\underline a= (a_j)_{j\in \mathbb N} \in \widehat \Sigma^+_{\mathfrak X}$ there exists an homeomorphism $h_{\underline a}: A_{\mathfrak X} \to A_{\mathfrak Y}$  so that 
    \begin{equation}
        \label{eq:conjtopconj}
        h_{\underline a}\circ f_{a_k}\circ\dots  \circ f_{a_1} \circ f_{a_0}
= g_{\tau(\underline a)_k}\circ \dots \circ g_{\tau(\underline a)_1}\circ  g_{\tau(\underline a)_0} \circ h_{\underline a}.
    \end{equation}
\end{definition}

Note that the previous notion of conjugacy assumes implicitly that the combinatorial structure of both attractors $A_{\mathfrak X}$ and $A_{\mathfrak Y}$ coincide (up to homeomorphism) and it concerns the 
stability of points points with infinite orbits. 

We are now in a position to define the concept of topological stability.

\begin{definition}
    {Let $(X,d)$ be a compact metric space and let $R_{\mathfrak X}=(X_i,f_i)_{1\le i \le n}$ be a local IFS. Given $*\in\{C,0\}$, we say that the local IFS $R_{\mathfrak X}$ is \emph{topologically stable in $\mathcal F_*$}  \index{Topologically stable} if there exists an open neighborhood $\mathcal V$ of $R_{\mathfrak X}$ in the metric space $(\mathcal F^{(n)}_*, D)$  
    so that $R_{\mathfrak Y}$
    is topologically conjugate to $R_{\mathfrak X}$ for every $ R_{\mathfrak Y} \in \mathcal V$.}
\end{definition}

Some comments are in order. The previous notion is stronger than the concept of topological stability introduced by Arbieto and Trilles \cite{AT25} for global IFSs. In the expansive setting considered there, the shadowing point is unique, whereas in the contractive setting this is generally not the case. Moreover, while \cite{AT25} requires only that the orbits of nearby systems remain uniformly close up to a $C^0$ change of coordinates near the identity, our formulation requires the existence of a homeomorphism that conjugates their orbits. 

Our next result ensures that, under the combinatorial stability assumption, the hyperbolicity of the local IFS is sufficient to imply its topological stability. 

\begin{maintheorem}
\label{thm:combinatorial-topological}
    Let $(X,d)$ be a compact metric space and assume that $R_{\mathfrak X}=(X_i,f_i)_{1\le i \le n}$ is a combinatorially stable local IFS. If $R_{\mathfrak X}$ is contractive and satisfies the open set condition then $R_{\mathfrak X}$ is topologically stable in $\mathcal F_C$.
\end{maintheorem}

As contractive IFSs satisfy the concordant shadowing property \cite{GV06}, the previous result 
implies that all contractive IFSs satisfying the open set condition are topologically stable (see Subsection~\ref{sec:top-stab-IFS} for the proof), a fact that seems not to have been proven before.

\begin{maintheorem}
\label{thm:concordantshadowing-implies-stable}
Let $(X,d)$ be a compact manifold of dimension  $\dim X\ge 3$ and that $X_j\subset X$ satisfy $X_j=\overline{\interior(X_j)}$ for every $1\le j\le n$.
Assume that $R_{\mathfrak X}= \left(X_j, f_j \right)_{1\le j \le n}$ is a contractive local IFS which is combinatorial stable and topologically stable in $\mathcal F_0$. Then 
$R_{\mathfrak X}$  satisfies the concordant shadowing property.
\end{maintheorem}
\color{black}

It remains an open question whether Theorem~\ref{thm:concordantshadowing-implies-stable} can be extended to local IFSs defined on more general metric spaces. The assumption that $X$ is a manifold (possibly with boundary) is crucial for the perturbation machinery employed in the proof.
We stress that, combining Theorems~\ref{thm:combinatorial-topological} and ~\ref{thm:concordantshadowing-implies-stable}  we obtain the following characterization of topological stability among the space of combinatorially stable, contractive local IFSs on manifolds.

\begin{maincorollary}
\label{cor:stable-iff-concordantshadowing}
Let $(X,d)$ be a compact manifold of dimension  $\dim X\ge 3$, $X_j\subset X$ 
satisfy $X_j=\overline{\interior(X_j)}$ for every $1\le j\le n$, and
$R_{\mathfrak X}= \left(X_j, f_j \right)_{1\le j \le n}$ be a combinatorially stable contractive local IFS. Then
$R_{\mathfrak X}$ is topologically stable if and only if it  satisfies the concordant shadowing property.
\end{maincorollary}

\subsection{Graph-directed IFSs} 

We finish this section with an application to graph directed IFSs \cite{MW88}, building over the fact that every graph-directed IFS can be realized as a contractive local IFS on an appropriate enriched compact metric space \cite{OV25a} as we now detail.
Let $(X,d)$ be a compact metric space and let
$(G, (X_v)_{v\in V}, (f_e)_{e \in E})$ 
be a contractive graph-directed IFS on $X$.
In \cite{OV25a} we proved that the latter is modeled by the contractive local IFS on the space $Y := X \times V$ 
(endowed with the metric
$d_Y\bigl((x,v),(y,w)\bigr) = d(x,y) + \mathbf{1}_{\{v\neq w\}}$) defined by
\begin{equation}
    \label{eq:gdext}
g_e : Y \to Y,
\qquad
g_e(x,v) = \bigl(f_e(x), i(e)\bigr)
\end{equation}
where
\begin{equation}
    \label{eq:gdext2}
D_e= X_{t(e)} \times \{t(e)\} \subset Y
\end{equation}
for each $e\in E$.
Moreover, every attractor of a contractive graph-directed IFS 
$$(G, (X_v)_{v\in V}, (f_e)_{e \in E})$$ 
on a compact metric space $(X,d)$ 
arises from a contractive local IFS defined on a compact subset of an enriched compact metric space $Y$ (cf. \cite[Theorem~G]{OV25a}).
Moreover, every graph-directed IFSs, so a
associated local IFS $\mathcal R_{\mathcal G}$ is combinatorially stable (cf. Remark~\ref{rmk:graph-directed-combin-stable}
and Proposition~\ref{thm:criterion-combinatorial-stability}).
Since $\mathcal R_{\mathcal G}$ is contractive then we derive the following immediate consequence from our previous stability results.

\begin{maincorollary}
\label{cor:graphdirected}
    Let $(X,d)$ be a compact metric space and let $\mathcal G=(G,(X_v)_{v\in V},(f_e)_{e\in E})$
be a contractive graph-directed IFS. If the associated contractive local IFS
$\mathcal R_{\mathcal G}=(D_e,g_e)_{e\in E}$
given by ~\eqref{eq:gdext} and \eqref{eq:gdext2}
satisfies the open set condition then $\mathcal R_{\mathcal G}$ is combinatorially stable and topologically stable in $\mathcal F_C$.
\end{maincorollary}

\begin{remark}
Note that $\mathcal R_{\mathcal G}$ being  topologically stable ensures that for every sufficiently small perturbation there exists a conjugating homeomorphism $H$ between the corresponding attractors. Since the fibers $X \times \{v\}$ are labeled by the vertices and $H$ preserves fibers, it can be writen as $H(x,v) = (h_v(x),v)$ for a family of homeomorphisms $(h_v)_{v\in V}$. The conjugacy relation implies
$h_{i(e)} \circ f_e = \tilde f_e \circ h_{t(e)}$
for every $e \in E$. 
\end{remark}
\color{black}

%%%%%%%%%%%%%%%%%%%%%%%%%%%%%%%%%
\section{Preliminaries}\label{sec:prelim}

\subsection{Code space}

In this section we collect some information obtained in \cite{OV25a}, namely on the code space  
and space of orbits of a contractive local IFS, which can be 
substantially more intricate than the context of IFSs.
Consider the space $\Sigma^-=\{1,2, \dots, n\}^{-\mathbb N}$ endowed with the metric
$$ 
d_-(\underline a, \underline b) = \exp\Big(\,-\sup\{N\ge 1\colon a_j = b_j, \;\text{for every}\; -N\le j \le -1 \}\Big)
$$ 
for every $\underline a, \underline b\in \Sigma^-$, where $\sup \emptyset =0$.

\begin{definition}\label{def: code map}
   Given a local IFS we define the \emph{local code map}\index{Local code map} $\pi_{\mathfrak X}: \Sigma^- \to 2^{X} $ as the set function
   $$\pi_{\mathfrak X}(\underline b)= \bigcap_{k \geq 1} V_{[\underline b]_k} = \lim_{k\to \infty} V_{[\underline b]_k},$$
   where
   $ 
   [\underline b]_k:= (b_{-k}, \dots, b_{-2}, b_{-1})
$ 
for each $\underline b\in \Sigma^-$ and $k\ge 1$,
$$ 
V_{[\underline b]_k} 
 =  f_{b_{-1}} \circ f_{b_{-2}} \circ \dots f_{b_{-k}}(X_{b_{-k}}),
$$  
and $f_{b_j}(A)=f_{b_j}(A\cap X_{b_j})$ for every integer $-k \le j \le -1.$   We define the \emph{local code space}\index{Local code space} as the set
    \begin{equation}
        \label{eq:def-localcodespace}
    \Sigma^-_{\mathfrak X}:=\{\underline b \in \Sigma^- \colon  \pi_{\mathfrak X}(\underline b) \neq \emptyset \} \subseteq \Sigma^-.
    \end{equation}
\end{definition}

The next results show that every contractive local IFS is semiconjugate to a contractive local (standard) IFS on a shift space, describe their code space and extended shift of orbits. 

\begin{theorem}\label{thm:classification}\cite[Theorem~A]{OV25a}
Let $(X,d)$ be a compact metric space and $X_j\subset X$, for $1\le j \le n$. If $R_{\mathfrak X}=(X_j,f_j)_{1\le j \le n}$ is a contractive local IFS then there exists a $\sigma^-$-invariant and compact shift space $\Sigma^-_{\mathfrak X} \subset \Sigma^{-}$, a H\"older continuous surjective map $\pi_{\mathfrak X}: \Sigma^-_{\mathfrak X} \to A_{\mathfrak X}$ and 
subsets $B_j=\pi_{\mathfrak X}^{-1}(X_j)\subset \Sigma^-_{\mathfrak X}$ so that:
\begin{enumerate}
    \item $\Sigma^-_{\mathfrak X}=\{\underline b \in \Sigma^- \colon  \pi_{\mathfrak X}(\underline b) \neq \emptyset \} \subseteq \Sigma^-$;
    \item $f_j\circ \pi(\underline b)=\pi_{\mathfrak X} \circ \tau_j(\underline b)$ for every $\underline b\in B_j$.
\end{enumerate}
In particular $R_{\mathfrak X}$ is semiconjugate to the local IFS $\mathcal S_{\mathfrak X}=(B_j,\tau_j)_{1\le j \le n}$. Moreover, if $R_{\mathfrak X}$  satisfies the open set condition then $\pi_{\mathfrak X}$ is a conjugacy.
\end{theorem}

Given a contractive local IFS  $R_{\mathfrak X}=(X_j,f_j)_{1\le j \le n}$ consider the extended shift space
\begin{equation}
    \label{eq:def-enlarged-Omega}
    \Omega =\Big\{((x_i,a_i))_{i\ge 0} \in (X\times \{1, 2, \dots, n\})^{\mathbb N} \colon x_{i+1}=f_{\underline a_i}(x_i) \; \text{for every}\; i\ge 0 \Big\}
\end{equation}
associated to points with infinite orbits
(it is implicitly assumed that $x_i\in X_{a_i}$ for every $i\ge 0$). 
Given $k\ge 0$, denote by $\pi_{k,X}: \Omega \to X$ the projection given by $\pi_{k,X}(((x_i,a_i))_{i\ge 0})=x_k$. 
Consider also the shift map $\bar \sigma: \Omega \to \Omega$ given by
\begin{equation}
    \label{eq:extendedsigma}
\bar \sigma(((x_0,a_0),(x_1,a_1), (x_2,a_2), \dots))= ((x_1,a_1), (x_2,a_2), (x_3,a_3) \dots)
\end{equation}
and the maximal invariant subset $A^{\infty}_{\mathfrak X}$, defined by ~\eqref{eq:maxinvA}. 

\begin{theorem}
\label{thmC}\cite[Theorem~C]{OV25a} 
Let $R_{\mathfrak X}=(X_j,f_j)_{1\le j \le n}$ be a contractive local IFS so that $A_{\mathfrak X}\neq \emptyset$. Then $A^\infty_{\mathfrak X}\neq \emptyset$. 
Furthermore, there exists a non-empty set $\widehat \Sigma_{\mathfrak X}\subset \Sigma^-_{\mathfrak X} \times \Sigma^+$ so that
\begin{enumerate}
    \item $(\pi_{\mathfrak X}\circ \pi_{\Sigma^-})(\widehat \Sigma_{\mathfrak X})=A^\infty_{\mathfrak X}$;
    \item $\widehat \Sigma_{\mathfrak X}$ is invariant by the two-sided shift $\widehat \sigma$ defined by 
    $$
\widehat \sigma ( \dots , b_{-3}, b_{-2}, b_{-1}, \boxed{a_0}, a_1, a_2, a_3,\dots)
    = 
     ( \dots , b_{-2}, b_{-1}, a_0, \boxed{a_1}, a_2, a_3, a_4, \dots);
$$
    \item 
    for each $(\underline b,\underline a) \in \widehat\Sigma_{\mathfrak X}$ if holds that
$$
\widehat \sigma(\underline b, \underline a)=(\underline b \ast a_0, \sigma(\underline a)) = (\tau_{a_0}(\underline b), \sigma(\underline a)).
$$
\end{enumerate}
\end{theorem}

\begin{remark}
Under the right encoding, the sequence $( \dots , b_{-3}, b_{-2}, b_{-1}, \boxed{a_0}, a_1, a_2, a_3,\dots)$ defines the dynamics of a local IFS in the following way: given $x_0=\pi_{\mathfrak X}(\dots , b_{-3}, b_{-2}, b_{-1}) \in A_{\mathfrak X}$ then one can obtain the next iterates as follows $x_1= f_{a_0}(x_0)$, $x_2= f_{a_1}(x_1)$, and so on. In this way, we form an orbit $((x_0,a_0), (x_1,a_1),\ldots)$ equivalent to $$( \dots , b_{-3}, b_{-2}, b_{-1}, \boxed{a_0}, a_1, a_2, a_3,\dots).$$ In the next, we will shift to this representation.    
\end{remark}

\begin{remark}
Given $x\in X$ consider the space 
$$
\Omega^x_\infty = \pi_{0,X}^{-1}(\{x\})= \Big\{ ((x_i,a_i))_{i\ge 0} \in \Omega_\infty \colon \pi_{0,X}(((x_i,a_i))_{i\ge 0})=x \Big\}
$$
which represents all possible orbits $(f_{a_{k}}\circ \ldots f_{a_{1}} \circ f_{a_{0}}(x))_{k\ge 0}$ of the point $x$, for all possible choices of paths $\underline a \in \Sigma^+$.
The space $\widehat \Sigma_{\mathfrak X}$ in Theorem~\ref{thmC} is obtained as 
\begin{equation}
    \label{defOmegax}
\widehat \Sigma_{\mathfrak X}=\Big\{(\underline b, \underline a) \in \Sigma^-_{\mathfrak X} \times \Sigma^+ \colon \underline a \in \pi_{\Sigma^+}(\Omega^x_\infty) \; \text{where} \; x=\pi_{\mathfrak X}(\underline b) \Big\}
\end{equation}
which, by construction, satisfies $(\pi_{\mathfrak X}\circ \pi_{\Sigma^-}) (\widehat \Sigma_{\mathfrak X})=A^\infty_{\mathfrak X}$.  \end{remark}

%%%%%%%%%%%%%%%%%%%%%%%%%%%%%%
\subsection{Combinatorial stability}

In this subsection, we initiate the study of the notion of combinatorial stability and its relation to the set of points with infinite orbits in the local attractor.
Recall the spaces of local IFSs and the $D$-distance introduced in~\eqref{def:distanceS}.

\begin{lemma}\label{le:past-2-future}
    Assume that $R_{\mathfrak X}$ is combinatorially stable. There exists an open neighborhood $\mathcal V$ of $R_{\mathfrak X}$ so that 
     $\widehat \Sigma_{\mathfrak Y}= \widehat \Sigma_{\mathfrak X}$ for every $R_{\mathfrak Y} \in \mathcal V$.
\end{lemma}

\begin{proof}
Consider an open neighborhood $\mathcal V$ of $R_{\mathfrak X}$ so that 
$\Sigma^-_{\mathfrak Y}=\Sigma^-_{\mathfrak X}$ for every $R_{\mathfrak Y} \in \mathcal V$. 
We claim that $\hat \Sigma_{\mathfrak Y}\subseteq \hat \Sigma_{\mathfrak X}$.
By equation~\ref{defOmegax},
$$
\widehat \Sigma_{\mathfrak Y}=\Big\{(\underline b, \underline a) \in \Sigma^-_{\mathfrak Y} \times \Sigma^+ \colon \underline a \in \pi_{\Sigma^+}(\Omega^y_\infty) \; \text{where} \; y=\pi_{\mathfrak Y}(\underline b), \; \underline b \in \Sigma^-_{\mathfrak Y} 
\Big\}.
$$ 
Hence, given $(\underline b,\underline a) \in \widehat \Sigma_{\mathfrak Y}$, the point $y=\pi_{\mathfrak Y}(\underline b)$ belongs to the attractor $A_{\mathfrak Y}$. As $\underline a \in \pi_{\Sigma^+}(\Omega^y_\infty)$ the orbit $\{g_{\underline a}^k(y)\colon k\ge 0\}$ of $y$ is well defined and contained in $A_{\mathfrak Y}$. In particular,  
$\underline b\ast (a_0,a_1, \dots, a_N) \in \Sigma^-_{\mathfrak Y}=\Sigma^-_{\mathfrak X}$ for every $N\ge 1$.
Therefore, 
$$
f_{a_N}\circ \dots \circ f_{a_1}\circ f_{a_0}(\pi_{\mathfrak X}(\underline b)) = \pi_{\mathfrak X}(\underline b\ast (a_0,a_1, \dots, a_N)) \in A_{\mathfrak X}
$$
for every $N\ge 1$.
Equivalently, $\underline a\in \pi_{\Sigma^+}(\Omega_\infty^x)$ where $x=\pi_{\mathfrak X}(\underline b)$, which proves that $(\underline b, \underline a)\in \widehat \Sigma_{\mathfrak X}$.
Since the reverse inclusion is identical, this finishes the proof of the lemma.
\end{proof}
\color{black}

There are examples of contractive local IFSs which satisfy the OSC condition but fail to be combinatorially stable (cf. proof of Theorem~\ref{thm:beta}). The next proposition provides a criterion for combinatorial stability of local IFSs.

\begin{proposition}
\label{thm:criterion-combinatorial-stability}
Let $(X,d)$ be a compact metric space and let
$R_{\mathfrak X}=(X_j,f_j)_{1\le j\le n}$ be a contractive local IFS. Assume that:
\begin{itemize}
    \item[(H1)] $R_{\mathfrak X}$ satisfies the open set condition;
    \item[(H2)] 
    either $
    f_j(X_j)\subset \interior(X_i)
    \;\text{or}\;
    f_j(X_j)\cap X_i=\emptyset
    $ 
    for every distinct $1\le i,j\le n$.
\end{itemize}
Then $\Sigma^- _{\mathfrak X}$ is a subshift of finite type and there exists an open neighborhood $\mathcal V$ of
$R_{\mathfrak X}$ in $\mathcal F_C$ such that
$\Sigma^-_{\mathfrak Z}=\Sigma^-_{\mathfrak X}$ 
for every $R_{\mathfrak Z}\in\mathcal V$.
\end{proposition}

\begin{proof}
Both assumptions (H1) and (H2) define open conditions in $\mathcal F_0$. In fact, by compactness of $X_i$ and $f_i(X_i)$ for every $1\le i \le n$, there exists $\zeta>0$ such that 
\begin{equation}
\label{eq:opencond1}
\min_{1\le i,j\le n} \dist_H\bigl(f_i(X_i),f_j(X_j)\bigr)\ge \zeta
\end{equation}
\begin{equation}
\label{eq:opencond2}
\min_{\{\substack{i,j} \colon f_i(X_i) \cap X_j=\emptyset\}} \Big\{
\mathfrak\dist_H\bigl(f_i(X_i),X_j\bigr)
\Big\}
\ge \zeta
\end{equation}
and 
\begin{equation}
\label{eq:opencond3}
\min_{\{\substack{i,j} \colon f_i(X_i) \subset int(X_j)\}} \Big\{
\mathfrak\dist_H\bigl(f_i(X_i),
X\setminus \mathrm{int}(X_j)\bigr)
\Big\}
\ge \zeta.
\end{equation}

Consider the transition matrix $M=(m_{ij})\in\mathcal M_{n\times n}(\{0,1\})$ defined by $m_{ij}=1$ if and only if $f_j(X_j)\cap X_i\neq\emptyset$.
By assumption (H2) it follows that 
$f_j(X_j)\subset  X_i$ whenever $m_{ij}=1$.

We claim that $\Sigma^-_{\mathfrak X}$ is a subshift of finite type. 
In fact $\underline b\in\Sigma^-_{\mathfrak X}$ if and only if
$$
V_{[\underline b]_k}
=f_{b_{-1}}\circ\cdots\circ f_{b_{-k}}(X_{b_{-k}}) 
$$
is non-empty
for every $k\ge1$. The latter is only possible when $f_{b_{-k}}(X_{b_{-k}})\cap X_{b_{-(k-1)}}\neq\emptyset$ for every $k\ge 2$ or, equivalently,
$m_{b_{-(k-1)},\,b_{-k}}=1$
for all  $k\ge2$.
Conversely, given $k\ge 2$ arbitrary, if $m_{b_{-(k-1)},\,b_{-k}}=1$
then
$f_{b_{-k}}(X_{b_{-k}})
\subset \interior(X_{b_{-(k-1)}})$ and, in consequence, 
the compact set $V_{[\underline b]_k}$ is non-empty for every $k\ge 1$ and 
$\pi_{\mathfrak X}(\underline b)
=
\bigcap_{k\ge1}V_{(b_{-k},\dots,b_{-1})}
\neq\emptyset$
(hence  $\underline b\in\Sigma^-_{\mathfrak X}$). Altogether we conclude that
\[
\Sigma^-_{\mathfrak X}
=
\Bigl\{
\underline b\in\Sigma^-:
m_{b_{-(k-1)},\,b_{-k}}=1
\text{ for every } k\ge2
\Bigr\},
\]
which proves that $\Sigma^-_{\mathfrak X}$ is a subshift of finite type.

\medskip
We now show that the local IFS $R_{\mathfrak X}$ is combinatorially stable. In fact, let $\zeta>0$ satisfy properties ~\eqref{eq:opencond1}-\eqref{eq:opencond3} and let 
$R_{\mathfrak Y}$ be a local IFS such that  $D(R_{\mathfrak X},R_{\mathfrak Y})<\zeta/3$.
This implies that 
$$
\dist_H(X_i,Y_i)<\zeta/3 
\quad\text{and}
\quad
\dist_H(f_i(X_i),g_i(Y_i))<\zeta/3 
$$
for every $1\le i \le n$. In particular, if $m_{ij}=1$  (equivalently $f(X_i)\subset \text{int}(X_j)$), one concludes that  
$$
\dist_H(g(Y_i),X\setminus Y_j)
\ge \dist_H(f(X_i),X\setminus X_j) 
-
\dist_H(g(Y_i), f(X_i))
-\dist_H(X\setminus X_j, X\setminus Y_j)
> \frac\zeta3.
$$
Similarly if $m_{ij}=0$ (or $f(X_i)\cap X_j=\emptyset$) then 
$$
\dist_H(g(Y_i),Y_j)
\ge \dist_H(f(X_i), X_j) 
-
\dist_H(g(Y_i), f(X_i))
-\dist_H(X_j, Y_j)
> \frac\zeta3.
$$
Therefore, the transition matrix $M$ also describes the admissible
transitions for $R_{\mathfrak Y}$, hence 
$ 
\Sigma^-_{\mathfrak Y}=\Sigma^-_{\mathfrak X}
$ 
is a subshift of finite type. We conclude that $R_{\mathfrak X}$ is combinatorially stable, as desired. 
\end{proof}

\begin{remark}
Assumption (H1) is necessary for the shadowing property  (cf. Example~\ref{ex:shift2}). \color{black}
Moreover, the fact that the code space above is a subshift of finite type arises naturally in the context of a Markov-type condition, such as assumption (H2).
\end{remark}

\subsection{Concordant shadowing}

In this subsection we will establish a quite useful relation on the concordant shadowing with respect to finite and infinite pseudo-orbits (cf. Lemma~\ref{le:finiteshadowequiv} below).
Given 
$\underline a\in\Sigma^+$, $\delta>0$ and a $(\underline a, \delta)$-pseudo-orbit, we say that the sequence $(x_k)_{k=0}^N$ 
is a \emph{$(\underline a,N,\delta)$-pseudo-orbit} if $d\bigl(f_{a_k}(x_k),x_{k+1}\bigr)<\delta$ for every $0\le k \le N-1$.

\begin{definition}\label{def:finite-shadowing}
Given a local IFS $R_{\mathfrak X}=(X_j,f_j)_{1\le j\le n}$,  we say that $R_{\mathfrak X}$ satisfies the \emph{finite concordant shadowing
property} if for every
$\vep>0$ there exists $\delta>0$ such that for every 
$\underline a\in\Sigma^+$
and $N\ge 1$ the following holds: for every $(\underline a,\delta)$-pseudo-orbit $(x_k)_{k\ge 0}$
there exists
$x\in X_{a_0}$ such that
\[
d\bigl(f_{\underline a}^k(x),x_k\bigr)<\vep
\quad\text{for all } 0\le k\le N.
\]
We say that the
$(\underline a,N,\delta)$-pseudo-orbit $(x_k)_{k=0}^N$ is \emph{$(\underline a,\vep)$-shadowed}
by $x$.
\end{definition}

\begin{remark}
    There is a contrast between the concept in Definition~\ref{def:finite-shadowing} and the usual notion of finite shadowing. In fact, while for a single map $f$ every finite pseudo-orbit $(x_k)_{k=0}^N$ can give rise to an infinite pseudo-orbit simply by taking $x_k=f^{k-N}(x_N)$ for all $k\ge N+1$, in the context of local IFSs there may exist finite pseudo-orbits (even finite orbits) to which one cannot a sequence in the space $\underline a\in \Sigma^+$. 
\end{remark}

\begin{lemma}
\label{le:finiteshadowequiv}
    Let $R_{\mathfrak X}=(X_j,f_j)_{1\le j \le n}$ be a local IFS.
Then $R_{\mathfrak X}$ has the concordant shadowing property 
if and only if it satisfies the finite concordant shadowing property.
\end{lemma}

\begin{proof}
By definition it is clear that the concordant shadowing property implies the finite concordant shadowing property.
It remains to prove the converse. 

Fix $\vep>0$ and let $\delta>0$ be given by the finite concordant shadowing
property.
Let $\underline a\in\Sigma^+$ and let $(x_k)_{k\ge 0}$ be a
$(\underline a,\delta)$-pseudo-orbit. Recall that $x_k\in X_{a_k}$ for all $k\ge 0$.
For each $N\ge 0$ consider the set
\[
E_N =
\Bigl\{
x\in X_{a_0}:\ d\bigl(f_{\underline a}^k(x),x_k\bigr)\le \vep
\ \text{for all } 0\le k\le N
\Bigr\}.
\]
By the finite concordant shadowing property applied to the finite pseudo-orbit
$(x_k)_{k=0}^N$
it follows that 
$E_N\neq\emptyset$ for every $N\ge 1$.
Since the family $(E_N)_{N\ge 1}$ is nested (by definition, $E_{N+1}\subset E_N$)
and each set $E_N\subset X_{a_0}$ is closed in $X_{a_0}$, hence compact, there exists a point
$ 
x \in \bigcap_{N\ge 0} E_N \neq \emptyset.
$ 
The pseudo-orbit $(x_k)_{k\ge 0}$ is $(\underline a,\vep)$-shadowed by $x$.
This proves the concordant shadowing property, as desired.   
\end{proof}

The previous lemma will be instrumental in the characterization of topological stability for local IFSs on compact manifolds.

\subsection{Shadowing property for skew-products}

In this subsection we offer a new insight to the concordant shadowing property, based on skew-product dynamics. It is well known that factors of maps with the shadowing property need not satisfy the shadowing property. The following simple lemma shows this is the case for the natural projection of a skew-product.

\begin{lemma}\label{projshdow}
Let $(Y,d_Y)$ and $(Z,d_Z)$ be compact metric spaces and endow the space $Y\times Z$ with the  metric 
 $d((y_1,z_1),(y_2,z_2)=\max\{d_Y(y_1,y_2), d_Z(z_1,z_2)\}.
$  Let  $S: Y \times Z \to Y \times Z $ be a continuous skew-product defined by
$
S(y,z) = (f(y), g(y,z)).
$ 
If $S$ satisfies the shadowing property then $f$ satisfies the shadowing property.
\end{lemma}

\begin{proof} Fix $\vep>0$ arbitrary. Let $\delta>0$ be given by the shadowing property for $S$.
We claim that every $\delta$-pseudo orbit for $f$ is $\vep$-shadowed by a true orbit of $f$. Indeed, let $ \{y_n\}_{n  \ge 0} \subset X $ be a $ \delta $-pseudo-orbit of $ f $, i.e.,
$d_Y(f(y_n), y_{n+1}) < \delta \quad \text{for all } n \ge 0.$
Fix an arbitrary point $ z_0 \in Z $. 
Consider the point $p_0 = (y_0, z_0)$.  We proceed recursively and define the sequence of point 
$ 
p_{n+1} = (y_{n+1}, g(p_n)), \quad \text{for each $n\ge 0$}
$ 
in $Y\times Z$. By construction 
\begin{align*}
d(S(p_n),p_{n+1}) 
    & = d((f(y_n), g(p_n)),(y_{n+1}, g(p_n)))
    =d(f(y_n),y_{n+1}) \\
& = {\max(d_Y(f(y_n),y_{n+1}), 0)} <\delta    
\end{align*}
for each $n\ge 1$, which guarantees that $(p_n)_{n\ge 0}$ is a $\delta$-pseudo orbit with respect to $S$.
By the shadowing property of $ S $, there exists $ (y, z) \in Y \times Z $ such that
$d(S^n(y, z), p_n) < \vep$ for all $n \ge 0.$
In particular, 
$ 
d_Y(f^n(x^*), x_n)  < \vep$ for all  $n \in \mathbb{N}$. This proves that $f$ has the shadowing property.
\end{proof}

The previous lemma allows one to relate concordant shadowing for local IFSs with the usual shadowing property for skew-products.
Let $\widehat \Sigma^+_{\mathfrak X}=\pi_{\Sigma^+}(\widehat \Sigma_{\mathfrak X})$, where 
$\widehat \Sigma_{\mathfrak X}$ is given by Theorem~\ref{thmC}.
Consider the space 
$$
\bigcup_{x\in A^\infty_{\mathfrak X}} \Big(\pi_{\Sigma^+}(\Omega^x_\infty) \times \{x\} \Big) \subset \Sigma^+_{\mathfrak X}\times A^\infty_{\mathfrak X}
$$
endowed with the metric
$d_S\bigl((\underline a,x),(\underline b,y)\bigr)
=\max\{d_\Sigma(\underline a,\underline b),\, d(x,y)\},$ 
and the skew-product
\[
S_{\mathfrak X}: \bigcup_{x\in A_{\mathfrak X}} \Big(\pi_{\Sigma^+}(\Omega^x_\infty) \times \{x\} \Big)\rightarrow
\bigcup_{x\in A_{\mathfrak X}} \Big(\pi_{\Sigma^+}(\Omega^x_\infty) \times \{x\} \Big)
\;\text{given by}\;
S(\underline a,x)=\bigl(\sigma(\underline a),\, f_{a_1}(x)\bigr).
\]

\begin{lemma}\label{equiv-shadowing}
Let $R_{\mathfrak X}=(X_i,f_i)_{1\le i\le n}$ be a contractive local IFS with local attractor
$A_{\mathfrak X}$. 
The following are equivalent:
\begin{enumerate}
\item the skew-product $S_{\mathfrak X}$ satisfies the shadowing property;
    \item $R_{\mathfrak X}$ satisfies the concordant shadowing property on $A^\infty_{\mathfrak X}$.
\end{enumerate}
\end{lemma}

\begin{proof}
We prove both implications separately.

\smallskip
$(1)\Rightarrow (2)$

Assume that $S_{\mathfrak X}$ has the shadowing property. 
Fix $0<\varepsilon<\frac12$ and let $\delta>0$ be given by the shadowing property.
Given any $\underline a\in \Sigma_{\mathfrak X}^+$ and any $(\underline a,\delta)$-pseudo-orbit $(x_k)_{k\ge 0}\subset A_{\mathfrak X}^\infty$
for the local IFS,
it is straightforward to check that the sequence $(\bigl(\sigma^k(\underline a),\,x_k\bigr))_{k\ge 0}$ in $\Sigma^+_{\mathfrak X}\times A_{\mathfrak X}$ 
is a $\delta$-pseudo-orbit for $S_{\mathfrak X}$. By the shadowing property for $S_{\mathfrak X}$, there exists
$(\underline{\tilde a},x)\in \Sigma^+_{\mathfrak X}\times A_{\mathfrak X}$ such that
\begin{equation}
    \label{eq:lemmaimpliskew1}
    d_S\bigl(
\bigl(\sigma^k(\underline{\tilde a}),\,f^k_{\underline{\tilde a}}(x_k)\bigr)
,\,\bigl(\sigma^k(\underline a),\,x_k\bigr)\bigr)
=
d_S\bigl(S_{\mathfrak X}^k(\underline{\tilde a},x),\,\bigl(\sigma^k(\underline a),\,x_k\bigr)\bigr)
<\vep
\end{equation}
for every $k\ge 0$.
In particular, 
$ 
d_\Sigma\bigl(\sigma^k(\underline{\tilde a}),\,\sigma^k(\underline a)\bigr)
\le \frac12
$ 
for each $k\ge 0$, which ensures that $\tilde a_{k+1}=a_{k+1}$ for each $k\ge 0$ and, consequently,
$\underline{\tilde a}=\underline a$.
Finally, from ~\eqref{eq:lemmaimpliskew1} and the definition 
of $d$, we also obtain
$d\bigl(f_{\underline{\tilde a}}^k(x),\,x_k\bigr)
=d\bigl(f_{\underline a}^k(x),\,x_k\bigr)
<\vep$ for each $k\ge 0$.
This proves that  $(x_k)_{k\ge 0}$ is $(\underline a,\varepsilon)$-shadowed by $x$. Thus
$R_{\mathfrak X}$ satisfies the concordant shadowing property.

\smallskip
\smallskip
$(2)\Rightarrow (1)$

Fix $\varepsilon>0$. Let $0<\delta<\frac12$ be such that it satisfies the requirements of the concordant shadowing property of the local IFS $R_{\mathfrak X}$ on the local attractor at scale $\vep$. 
It is clear that any $\delta$-pseudo-orbit $((\underline a^{(k)},x_k))_{k\ge 0}$  on  
$\Sigma^+_{\mathfrak X}\times A_{\mathfrak X}$  for $S_{\mathfrak X}$ is such that 
\begin{align}
\label{eq:base-close}
d_\Sigma\bigl(\sigma(\underline a^{(k)}),\,\underline a^{(k+1)}\bigr) <\frac12,
\quad \text{and}\quad 
d\bigl(f_{a^{(k)}_1}(x_k),\,x_{k+1}\bigr) <\delta
\end{align}
for every $k\ge 0$. The first inequality above implies that $a^{(k)}_2=a^{(k+1)}_1$ for every $k\ge 0$.
In other words, the sequence $(\underline a^{(k)})$ in $\Sigma^+_{\mathfrak X}$ is a true orbit, hence
$\underline a^{(k)}=\sigma^k(\underline a^{(0)})$
for every $k\ge 0.$ 
For notational simplicity we shall write simply $\underline a=\underline a^{(0)}$. Then the second inequality in ~\eqref{eq:base-close} becomes
$ 
d\bigl(f_{a_{k+1}}(x_k),\,x_{k+1}\bigr)<\delta
$ for all $k\ge 0,$ 
meaning that $(x_k)_{k\ge 0}$ is a $(\underline a,\delta)$-pseudo-orbit for the local IFS, hence 
$(\underline a,\vep)$-shadowed by some point $x\in A_{\mathfrak X}$.
Therefore, 
\[
d_S\bigl(S^k(\underline a,x),\,(\underline a^{(k)},x_k)\bigr)
=
d_S\bigl((\underline a^{(k)},f_{\underline a}^k(x)),\,(\underline a^{(k)},x_k)\bigr)
=
d\bigl(f_{\underline a}^k(x),x_k\bigr)
<\vep
\]
for every $k\ge 0$.
We conclude that every $\delta$-pseudo-orbit of $S_{\mathfrak X}$ is $\vep$-shadowed, and so $S_{\mathfrak X}$ satisfies the shadowing property. This finishes the proof of the lemma.
\end{proof}

\subsection{On the structure of pseudo-orbits}
\label{sec:nearbyPO}

In this subsection, we prove an instrumental result, that pseudo-orbits for a local IFS are shadowed by pseudo-orbits of nearby local IFSs. This is a non-trivial fact due to the restriction on the domains, which follows as a consequence of a deeper result: with respect to the metric $D$ on the space $\mathcal F$, every local IFS sufficiently close to $R_{\mathfrak X}=(X_j,f_j)_{1\le j \le n}$ determines pseudo-orbits relative to $R_{\mathfrak X}$.
Recall the role of the constant $L>0$ in ~\eqref{eq:distanceS}.

\begin{proposition}
    \label{prop.key}
    Let $R_{\mathfrak X}=(X_j,f_j)_{1\le j \le n}$ be a contractive local IFS. For every $0<\delta< L$ and each contractive local IFS $R_{\mathfrak Y}=(Y_j,g_j)_{1\le j \le n}$ which is $\delta$-close to $R_{\mathfrak X}$ the following holds: 
    if $(y_{{i}})_{i \geq 0}$ is a $(\underline a,\delta)$-pseudo orbit for $R_{\mathfrak Y}$ then there exists an $(\underline a,5\delta)$-pseudo orbit $(x_{{i}})_{i \geq 0}$ for $R_{\mathfrak X}$ satisfying
     $d(x_{{j}},  y_{{j}})< \delta$ for $j\geq 0$. 
\end{proposition}

\begin{proof}
Fix $0<\delta<L$.
    
    \begin{figure}[H]
        \centering
        \includegraphics[width=0.7\linewidth]{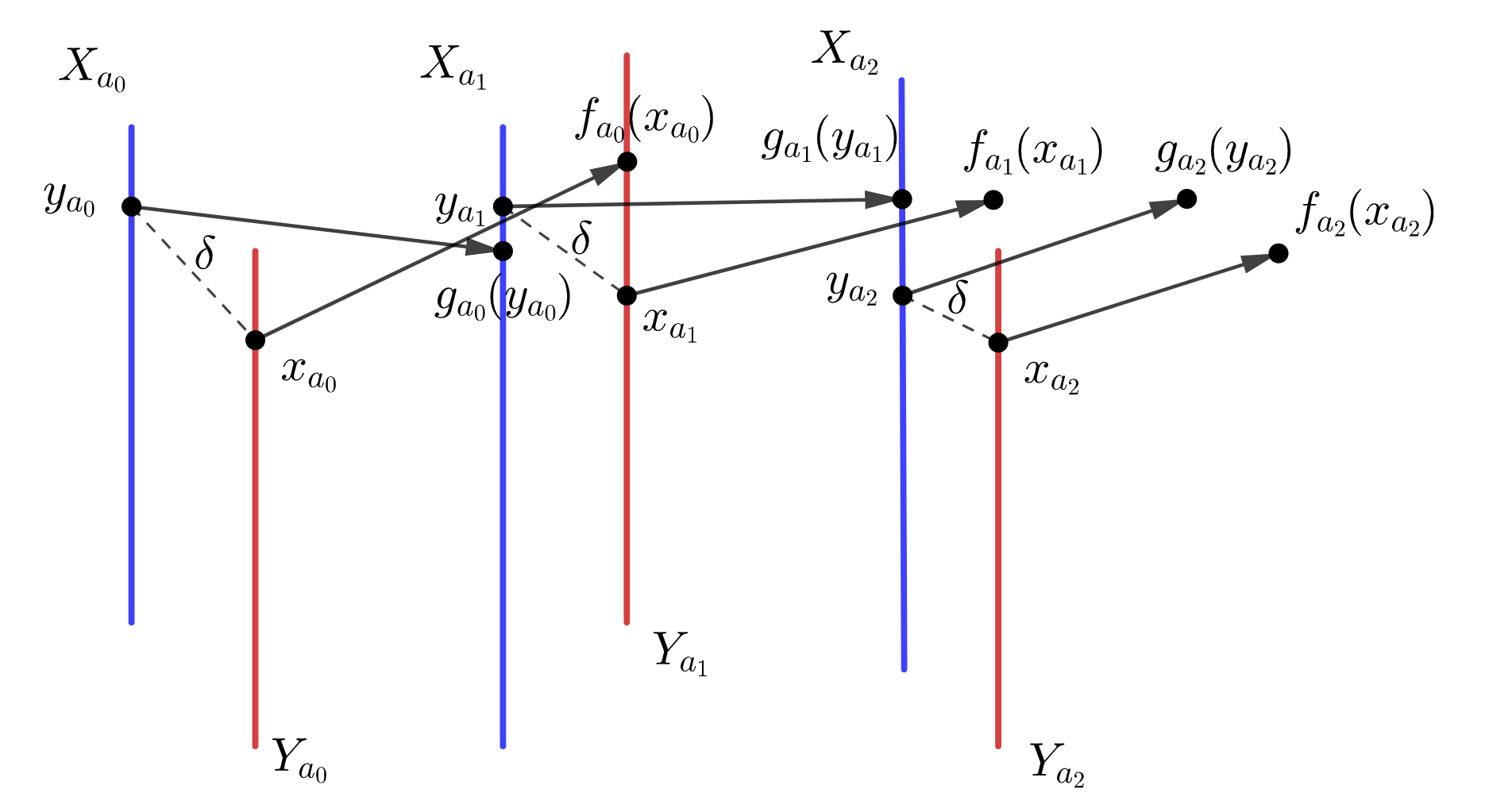}
        \caption{Shadowing scheme}
        \label{fig:shadowing}
    \end{figure}
    
    Let $(y_{{i}})_{i \geq 0}$ be a $(\underline a,\delta)$-pseudo orbit for $R_{\mathfrak Y}$. Then,
    \[d(y_{{j+1}},g_{{j}}(y_{{j}}))<\delta \;\quad \text{for every }\,  j \geq 0 \text{ and } y_{{j}} \in Y_{a_{j}}.\]
    Suppose that $R_{\mathfrak Y}$ is $\delta$-close to $R_{\mathfrak X}$. Hence $dist_H(X_{a_j},Y_{a_j})< \delta$ and, consequently, for each $j\ge 0$ there exists a point $  x_{{j}} \in X_{a_{j}}$ with 
    \begin{equation}
        \label{eq:esstimhaus}
        d(  x_{{j}},  y_{{j}})< \delta.
    \end{equation}
    By triangular inequality, 
    \begin{equation}
        \label{eq:236789}
        d(f_{a_{0}}( x_{{0}}),  x_{{1}}) \leq d(f_{a_{0}}(x_{{0}}), g_{a_{0}}( y_{{0}})) +d( g_{a_{0}}( y_{{0}}), y_{{1}}) +d( y_{{1}},   x_{{1}}).
    \end{equation}
    By construction, $d( y_{{1}},   x_{{1}})<\delta$ and $d( g_{a_{0}}( y_{{0}}), y_{{1}})<\delta$.   
As $D(R_{\mathfrak X},R_{\mathfrak Y})<\delta$ there exists an homeomorphism 
$\zeta_{a_{0}}: Y_{a_{0}}\to X_{a_{0}}$ so that $d_{C^0}(\zeta_{a_0}, id_{Y_{a_0}})<\delta$
and $d_{C^0}(f_{a_0}\circ \zeta_{a_0}, g_{a_0})<\delta$.
Then one can write $x_{0}=\zeta_{a_0}(\tilde y_0)$ for some $\tilde y_0 \in Y_{a_0}$ and, consequently,
$$
d(y_0, \tilde y_0) \le d(x_0, \tilde y_0)+ d(x_0, y_0) \le d_{C^0}(\zeta_{a_0}, id_{Y_{a_0}})+\delta <2\delta
$$
and
\begin{align*}
d(f_{a_{0}}(x_{{0}}), g_{a_{0}}( y_{{0}}))
& = d(f_{a_{0}}(\zeta_{a_0}(\tilde y_{{0}})), g_{a_{0}}( y_{{0}})) \\
& \le 
d(f_{a_{0}}(\zeta_{a_0}(\tilde y_{{0}})), g_{a_{0}}( \tilde y_{{0}}))
+ 
d(g_{a_{0}}( \tilde y_{{0}}), g_{a_{0}}( y_{{0}})) \\
& \le \delta + 2\lambda \delta  < 3\delta 
\end{align*}
(here $\lambda\in (0,1)$ is a contraction rate for $R_{\mathfrak Y}$).
Together with ~\eqref{eq:236789} we obtain that $d(f_{a_{0}}( x_{{0}}),  x_{{1}})<5\delta$.

The proof follows by applying the same argument used recursively - using ~\eqref{eq:esstimhaus} and that 
$x_{j}=\zeta_{a_j}(\tilde y_j)$ for some $\tilde y_j\in Y_{a_j}$ and some homeomorphism $\zeta_{a_j}$ that is $C^0$-close to the identity - in order to produce a $(\underline a,5 \delta)$-pseudo orbit for $R_{\mathfrak X}$. 
\end{proof}

Proposition~\ref{prop.key} yields the following consequence.

\begin{corollary}\label{cor:shadowclose}
    If the contractive local IFS $R_{\mathfrak X}=(X_j,f_j)_{1\le j \le n}$ satisfies the concordant shadowing property then for every $\vep>0$ there exists $\delta>0$ so that for each contractive local IFS $R_{\mathfrak Y}=(Y_j,g_j)_{1\le j \le n}$ which is $\delta$-close to $R_{\mathfrak X}$, every $(\underline a,\delta)$-pseudo orbit of $R_{\mathfrak Y}$ is $(\underline a,\vep)$-shadowed by an orbit of $R_{\mathfrak X}$.
    In particular $\hat \Sigma_{\mathfrak X}^+ \supseteq \hat \Sigma_{\mathfrak Y}^+$.
\end{corollary}

\begin{proof}
    Fix $0<\vep<L$. By the concordant shadowing property for $R_{\mathfrak X}$, that is,  for every 
    $\underline a \in \hat \Sigma_{\mathfrak X}$ and $0< \eta< \vep/2$ there exists $0<\delta<\vep/10$ such that all $(\underline a,5\delta)$-pseudo orbits for $R_{\mathfrak X}$ are $(\underline a,\eta)$-shadowed. 
Then
    Proposition~\ref{prop.key} implies that if $(y_{{i}})_{i \geq 0}$ is a $(\underline a,\delta)$-pseudo orbit for $R_{\mathfrak Y}$ then there exists a $(\underline a,5\delta)$-pseudo orbit $(x_{{i}})_{i \geq 0}$  for $R_{\mathfrak X}$ such that $d(x_{{j}},  y_{{j}})< \delta$ for every $j\geq 0$.
    Hence, there exists $x \in X_{a_{0}}$ such that 
    \[d(f_{\underline a}^k(x), x_{{k}}) < \eta, \; \forall k \geq 1,\]
    where we adopted the notation $f_{\underline a}^k(x):=f_{a_{k-1}}(f_{a_{k-2}}(\cdots(f_{a_{0}}(x))))$ for $k \geq 1$.
    Moreover,
    \[d(f_{\underline a}^k(x), y_{{k}})\leq  d(f_{\underline a}^k(x), x_{{k}}) + d(x_{{k}}, y_{{k}}) < \eta +5\delta<\vep
    \]
    for every $k \geq 1,$ thus proving the corollary.
\end{proof}

\begin{remark}
\label{rmk:usc}
    By Corollary~\ref{cor:shadowclose}, the map ${\mathfrak X} \mapsto \hat \Sigma^+_{\mathfrak X}$ is upper-semicontinuous (in Vietoris topology) at all $\mathfrak X$ so that the  local IFS $R_{\mathfrak X}$ has concordant shadowing property. In consequence, if the map is not upper semicontinuous at $\mathfrak X$ then $R_{\mathfrak X}$ does not satisfy the concordant shadowing property.
\end{remark}

We also obtain the following immediate consequence of Corollary~\ref{cor:shadowclose}.

\begin{corollary}
\label{cor:rob-shadow+}
    If $R_{\mathfrak X}=(X_j,f_j)_{1\le j \le n}$ satisfies the robust concordant shadowing property then there exists an open neighborhood $\mathcal V$ of $R_{\mathfrak X}$ so that 
    $\hat \Sigma_{\mathfrak Y}^+ = \hat \Sigma_{\mathfrak X}^+$ for every $R_{\mathfrak Y} \in \mathcal V$.
\end{corollary}

%%%%%%%%%%%%%%%%%%%%%%%%%%%%%%%%%
\section{Semicontinuity of the attractor and the code space }\label{sec:usc}

This section is devoted to the proof of
Theorems~\ref{thm:shadowing-usc} and ~\ref{thm:usc}.

\subsection{Proof of Theorem~\ref{thm:shadowing-usc}}
The proof builds over the rich structure of pseudo-orbits (recall Subsection~\ref{sec:nearbyPO}).
Assume that $R_{\mathfrak X}=(X_j,f_j)_{1\le j \le n}$ satisfies the concordant shadowing property. 
Then, 
for any $\vep>0$ there exists $\delta=\delta_\vep>0$ such that
letting $\mathcal V_\delta$ denote the $\delta$-neighborhood of $R_{\mathfrak X}$, the following property holds: 
for each $R_{\mathfrak Y}=(Y_j,g_j)_{1\le j \le n}$ in $\mathcal V_\delta$, every $(\underline a,\delta)$-orbit of 
$R_{\mathfrak Y}$ is $(\underline a,\delta)$-shadowed by an orbit of $R_{\mathfrak X}$
(recall Corollary~\ref{cor:shadowclose}).

One knows that $V_{[\underline b]_k}(Y_{b_{-k}}) \neq \emptyset$ for every $k\ge 1$ and every $\underline b\in \Sigma^-_{\mathfrak Y}$. In particular, there exists $z_k\in Y_{b_{-k}}$ 
so that
$$
y_k=g_{b_0} \circ g_{b_{-1}} \circ \dots \circ g_{b_{-k}}(z_k) \in Y_{b_0}.
$$
By the assumptions, there exists $x_k \in X_{b_{-k}}$ that $(\underline a,\vep)$-shadows the previous orbit, and so
$ 
f_{b_0} \circ f_{b_{-1}} \circ \dots \circ f_{b_{-k}}(x_k) \in X_{b_0}.
$ 
This ensures that  $V_{[\underline b]_k}(X_{b_{-k}}) \neq \emptyset$ for every $k\ge 1$.  As $\underline b\in \Sigma^-_{\mathfrak Y}$ was chosen arbitrary, this proves that
$\Sigma_{\mathfrak Y}^- \subset \Sigma_{\mathfrak X}^-$.
\smallskip

Now, we use the maximality of $\Sigma_{\mathfrak X}^-$ to prove the upper semicontinuity of the local attractor at $R_{\mathfrak X}$. In fact, notice that for each $k\ge 0$ the map 
$$
U_n \colon  R_{\mathfrak Y} \mapsto \bigcup_{\underline b\in \Sigma^-_{\mathfrak Y}} V_{[\underline b]_k}(Y_{b_{-k}})
$$
is upper-semicontinuous at $R_{\mathfrak X}$, by maximality of $\Sigma_{\mathfrak X}^-$ and continuity of the set functions $R_{\mathfrak Y} \mapsto V_{[\underline b]_k}(Y_{b_{-k}})$, for each $\underline b\in \Sigma_{\mathfrak Y}^-$.
This ensures that the local attractor map
$$
U: R_{\mathfrak Y} \mapsto A_{\mathfrak Y}=\bigcap_{k\ge 0} \bigcup_{\underline b\in \Sigma^-_{\mathfrak Y}} V_{[\underline b]_k}(Y_{b_{-k}})
= \inf_{n\ge 1} U_n(R_{\mathfrak Y})
$$
is upper-semicontinuous at $R_{\mathfrak X}$. This finishes the proof of the theorem. \hfill $\square$

\subsection{Proof of Theorem~\ref{thm:usc}}

Let $(X,d)$ be a compact metric space and let $(f_j)_{1\le j \le n}$ be a finite family of self maps on $X$ so that $f_j$ is a contraction for each $1\le j \le n$.
We aim at proving that the local attractor map $(K^*(X))^n \ni \mathfrak  X    \mapsto A_{\mathfrak X}$
	is upper-semicontinuous in the Vietoris topology: 
    given $\mathfrak  X \in (K^*(X))^n$ and $\vep>0$ there exists $\delta>0$ such that if $\max_{1\leq j \leq n} \operatorname{dist}_H(Y_j, X_j) < \delta$ then $A_{\mathfrak Y} \subset A_{\mathfrak X}^{\varepsilon}$. 
	
    Suppose, by contradiction, this is not the case. Then, there are $\mathfrak  X \in (K^*(X))^n$ and $\vep>0$ such that for any $m\geq 1$, there exists a collection of compact sets
$\mathfrak{ Y}_m \in (K^*(X))^n$ such that 
\begin{equation}
    \label{eq.contradii}
\max_{1\leq j \leq n} \operatorname{dist}_H(Y_j^m, X_j) < 1/m \quad \text{but} \quad   A_{\mathfrak Y_m} \not\subset A_{\mathfrak X}^{\varepsilon}.
\end{equation}
The second expression in  ~\eqref{eq.contradii} ensures that there exists $x_m \in A_{\mathfrak{ Y}_m}$ with $d(x_m,  A_{\mathfrak X}) \geq \varepsilon$.
Consider a sequence $\underline b^m = (\ldots, b_{-2}^{m},  b_{-1}^{m})$ in the code space $\Sigma_{{\mathfrak Y}_m}^-\subset \Sigma^-=\{1,2,\dots, n\}^{-\mathbb N}$ such that $\pi_{{\mathfrak Y}_m}(\underline b^m)=x_m$. In particular, 
\begin{equation}\label{ex:xnaprox}
  x_m =  \lim_{k\to \infty } f_{b_{-1}^{m}} \circ f_{b_{-2}^{m}} \circ \cdots \circ f_{b_{-k}^{m}} (y_{k}^{m}),  
\end{equation}
for some $y_{k}^{m}\in Y_{b_{-k}^{m}}^m$.

\begin{lemma}\label{le:primeiraaprox}
There exists a strictly increasing sequence
$(m_i)_{i\ge 1}$ and a sequence
$\underline b^*=(\ldots,b_{-2}^*,b_{-1}^*)\in \Sigma^-$ such that, given $i\ge 1$ one has that $b_{-j}^{m_i}=b_{-j}^*$ for any $1\le j\le i$.
\end{lemma}

\begin{proof}
    Since each symbol $b_{-1}^m$ belongs to the set $\{1,2,\dots,n\}$, there exists a symbol $b_{-1}^*$ and a strictly increasing
subsequence $(m_i^{(1)})_{i\ge 1}$ such that
$b_{-1}^{m_i^{(1)}}=b_{-1}^*$ 
for every  $i\ge 1$. Applying the same argument,  there exists
$b_{-2}^*\in \{1,2,\dots,n\}$ and a strictly increasing subsequence
$(m_i^{(2)})_{i\ge 1}$ of $(m_i^{(1)})_{i\ge 1}$ such that
$b_{-2}^{m_i^{(2)}}=b_{-2}^*$ for every  $i\ge 1$. Proceeding recursively, we obtain a decreasing sequence of strictly increasing
subsequences
\[
(m_i^{(1)}) \supset (m_i^{(2)}) \supset (m_i^{(3)}) \supset \cdots
\]
and symbols $b_{-1}^*, b_{-2}^*, b_{-3}^*, \dots \in \{1,2,\dots,n\}$ in such a way that
\[
b_{-j}^{m_i^{(r)}}=b_{-j}^*
\quad \text{for every } 1\le j\le r
\]
for every $i\ge 1$ and $r\ge 1$.
Now the proof explores Cantor's diagonalization process. Consider the diagonal sequence 
$(m_i)_{i\ge 1}$ given by $m_i:=m_i^{(i)}$
and note that 
\[
b_{-j}^{m_i}=b_{-j}^*
\quad \text{for every } 1\le j\le i
\]
and every $i\ge 1$.  This proves the lemma.
\end{proof}

The previous lemma plays a crucial role for the next key step in the proof of Theorem~\ref{thm:usc}.

\begin{lemma}\label{le:diagn}
Let $\underline b^*\in \Sigma^-$ denote the sequence determined by Lemma~\ref{le:primeiraaprox}. Then $\underline b^*\in \Sigma_{\mathfrak X}^-$.
\end{lemma}

   \begin{proof}
Fix $k\ge 1$. We need to prove that the finite word $(b_{-k}^*,\ldots,b_{-2}^*,b_{-1}^*)$ 
is admissible for iterations of $R_{\mathfrak X}$.
By Lemma~\ref{le:primeiraaprox}, for every $i\ge k$ one has
\[
b_{-j}^{m_i}=b_{-j}^*
\quad \text{for every } 1\le j\le k.
\]
Let us use the information about the code spaces associated to the approximate domains ${\mathfrak Y}_{m_i}$. Fix $i\ge k$. Since $\underline b^{m_i}\in \Sigma_{\mathfrak Y_{m_i}}^-$ one obtains that
\[
Y_{b_{-k}^*}^{m_i}
\cap
f_{b_{-k}^*}^{-1}\left(Y_{b_{-k+1}^*}^{m_i}\right)
\cap \cdots \cap
\left(f_{b_{-1}^*}\circ\cdots\circ f_{b_{-k}^*}\right)^{-1}(X)
\neq \emptyset.
\]
In other words, there exists $y_k^i\in Y_{b_{-k}^*}^{m_i}$ such that
\[
f_{b_{-1}^*}\circ\cdots\circ f_{b_{-k}^*}(y_k^i)
\in Y_{b_{-1}^*}^{m_i}.
\]
Since $X$ is compact, up to considering a subsequence we assume without loss of generality 
that $(y_k^i)_{i\ge 1}$ converges to some point $y_k\in X$.
Moreover, as
\begin{align*}
    \text{dist}_H(y_k,X_{b_{-k}^*})
    & \le 
\text{dist}_H(y_k,y_k^i) + 
\text{dist}_H(y_k^i,X_{b_{-k}^*}) \\
& \le \text{dist}_H(y_k,y_k^i) + 
\text{dist}_H(Y_{b_{-k}^*}^{m_i},X_{b_{-k}^*}),
\end{align*}
and the right-hand side tends to zero as $i\to +\infty$, one concludes that $y_k \in X_{b_{-k}^*}$. By continuity of the maps $f_j$, 
\[
f_{b_{-\ell}^*}\circ\cdots\circ f_{b_{-k}^*}(y_k)
=
\lim_{i\to\infty}
f_{b_{-\ell}^*}\circ\cdots\circ f_{b_{-k}^*}(y_k^i)
\]
and the same argument as above ensures that 
$f_{b_{-\ell}^*}\circ\cdots\circ f_{b_{-k}^*}(y_k)
\in X_{b_{-\ell}^*}.$ 
This proves the lemma.
\end{proof}

We are now in a position to complete the proof of Theorem~\ref{thm:usc}. By compactness, up to consider a subsequence we may assume that the sequence $(x_{m_i})_{i\ge 1}$ given by ~\eqref{ex:xnaprox} converges to a point $x\in X$. We claim that $x=\pi_{\mathfrak X}(\underline b^*)$. First note that as
$$
x_{m_i} =  \lim_{k\to \infty } f_{b_{-1}^{m_i}} \circ f_{b_{-2}^{m_i}} \circ \cdots \circ f_{b_{-k}^{m_i}} (y_{k}^{m_i})
$$
and 
\[
b_{-j}^{m_i}=b_{-j}^*
\quad \text{for every } 1\le j\le k \;\text{and} \; i\ge k
\]
and each map $f_j$ is a contraction, 
then denoting by $\lambda\in (0,1)$ a common contraction rate one concludes that 
\begin{align*}
d(x_{m_i}, f_{b_{-1}^*}\circ\cdots\circ f_{b_{-k}^*} (y_k^{m_i}))  & = d\left(
x_{m_i},
f_{b_{-1}^{m_i}}\circ\cdots\circ f_{b_{-k}^{m_i}}
(y_k^{m_i})
\right) \\
& \le \diam(f_{b_{-1}^{m_i}}\circ\cdots\circ f_{b_{-k}^{m_i}}
(Y_{b_{-k}^{m_i}})) \le \lambda^k\diam(X)
\end{align*}
for any fixed $1\le j \le k \le i$. 
Moreover, using that $y_k^{m_i}\in Y_{b_{-k}^{m_i}}^{m_i}=Y_{b_{-k}^*}^{m_i}$ and
$\text{dist}_H(Y_{b_{-k}^*}^{m_i},X_{b_{-k}^*})$ tends to zero as $i\to+\infty$
one can choose \(z_k^i\in X_{b_{-k}^*}\) such that
$d(y_k^{m_i},z_k^i)\le \text{dist}_H(Y_{b_{-k}^*}^{m_i},X_{b_{-k}^*})$. 
By the contraction of the maps $f_j$ we deduce that 
\[
d\left(
f_{b_{-1}^*}\circ\cdots\circ f_{b_{-k}^*}(y_k^{m_i}),
f_{b_{-1}^*}\circ\cdots\circ f_{b_{-k}^*}(z_k^i)
\right)
\le
\lambda^k d(y_k^{m_i},z_k^i) \le \lambda^k\diam(X).
\]
On the other hand, as $\underline b^*\in \Sigma_{\mathfrak X}^-$ (recall Lemma~\ref{le:diagn}) we have
\[
\pi_{\mathfrak X}(\underline b^*)
=
\lim_{\ell\to\infty}
f_{b_{-1}^*}\circ\cdots\circ f_{b_{-\ell}^*}(w_\ell)
\]
for any admissible point $w_\ell\in X_{b_{-\ell}^*}$ , and so
\[
d\left(
\pi_{\mathfrak X}(\underline b^*),
f_{b_{-1}^*}\circ\cdots\circ f_{b_{-k}^*}(z_k^i)
\right)
\le \lambda^k \diam(X_{b_{-k}^*})
\le \lambda^k\diam(X).
\]
Altogether we deduce that 
\begin{align*}
d(x_{m_i},\pi_{\mathfrak X}(\underline b^*))
&\le 3 \lambda^k\diam(X) \quad\text{for each $i\ge k$}.
\end{align*}
Taking the limsup as $i$ tends to infinity followed by the limit as $k$ tends to infinity one concludes that $x=\pi_{\mathfrak X}(\underline b^*),$ thus proving the claim. 

Note that the latter ensures that $x\in A_{\mathfrak X}$. However, this leads to a contradiction as the assumption that 
$d(x_m,A_{\mathfrak X})\ge \varepsilon$
for each $m\ge 1$ implies $d(x,A_{\mathfrak X})\ge \varepsilon>0$.
Altogether this proves that the local attractor map $(K^*(X))^n\ni \mathfrak X\mapsto A_{\mathfrak X}$ 
is upper-semicontinuous in the Vietoris topology and finishes the proof of the theorem.

\section{Shadowable pseudo-orbits and shadows}\label{sec:shadows}

In this section we prove Theorem~\ref{thm:Y}, on the concordant shadowing property and shadowable pseudo-orbits for significant sequences in $\Sigma_\infty$.

\subsection{Shadowable sequences and points}

	Let $R_{\mathfrak X}=(X_j,f_j)_{1\le j \le n}$ be a contractive local IFS.	
	By construction, $I(\underline{a})$ is a subset of the compact set $X_{a_0}$ formed by points in $X_{a_0}$ which admit an infinite orbit determined by $\underline a$ (recall~\eqref{eq:def.Ia}).
	The following lemma provides an optimal approximation of a pseudo-orbit by an actual orbit of the local IFS.
	
	\begin{lemma}\label{lem:L1}
		Let $\underline{a} \in \Sigma_{\infty}$ and $x \in I(\underline{a})$. Then, for any $(\underline{a}, \delta)$-pseudo-orbit $(x_i)_{i \geq 0}$, we have:
		\[ d(f_{\underline{a}}^m(x), x_m) \leq \frac{1}{1-\lambda} \delta + \lambda^m d(x, x_0), \quad \forall m \geq 0. \]
	\end{lemma}
	
	\begin{proof}
		The proof proceeds by induction on $m$, proving that 
        $$
        d(f_{\underline{a}}^m(x), x_m) \leq (1+\lambda +\dots +\lambda^m) \delta + \lambda^m d(x, x_0)
        $$ for every $m\ge 1$.
        Since $x \in I(\underline{a})$, it can be consecutively iterated by $f_{a_0}, f_{a_1}, \dots$. For $m=1$, by triangular inequality,
		\[ d(f_{\underline{a}}^1(x), x_1) \leq d(f_{a_0}(x), f_{a_0}(x_0)) + d(f_{a_0}(x_0), x_1) \leq \lambda d(x, x_0) + \delta. \]
        Assuming that  
$d(f_{\underline{a}}^m(x), x_m) \leq (1+\lambda +\dots +\lambda^m) \delta + \lambda^m d(x, x_0)$ for some $m\ge 1$, one obtains
\begin{align*}
d(f_{\underline{a}}^{m+1}(x), x_{m+1})
&= d\bigl(f_{a_m}(f_{\underline{a}}^m(x)), x_{m+1}\bigr) \\
&\le d\bigl(f_{a_m}(f_{\underline{a}}^m(x)), f_{a_m}(x_m)\bigr)
   + d\bigl(f_{a_m}(x_m), x_{m+1}\bigr) \\
&\le \lambda\, d(f_{\underline{a}}^m(x), x_m) + \delta \\
&\le \lambda\Big((1+\lambda+\cdots+\lambda^{m-1})\delta+\lambda^m d(x,x_0)\Big)+\delta \\
&= (1+\lambda+\cdots+\lambda^m)\delta+\lambda^{m+1}d(x,x_0),
\end{align*}
This proves the inductive step, and the lemma follows immediately.
	\end{proof}
	
	We now provide a simple  obstruction to the shadowing property for local IFSs. 
	
	\begin{lemma}\label{lem:L2}
		Let $\underline{a} \in \Sigma_{\infty}$, let  
        $(y_i)_{i \geq 0}$ be a $(\underline{a}, \delta)$-pseudo-orbit  and assume that $$\varepsilon_0 = \operatorname{dist}_H(\{y_0\}, I(\underline{a})) > 0.$$ If $0<\varepsilon < \varepsilon_0$, then $(y_i)_{i \geq 0}$ is not $(\underline{a}, \varepsilon)$-shadowed by any orbit of $R_{\mathfrak{X}}$.
	\end{lemma}
	
	\begin{proof}
		Suppose, by contradiction, that $(y_i)_{i \geq 0}$ is $(\underline{a}, \varepsilon)$-shadowed by some orbit of $R_{\mathfrak{X}}$. Then there exists a point $x \in I(\underline{a})$ such that $d(x, y_0) < \varepsilon < \varepsilon_0$, which contradicts the definition of $\varepsilon_0$.
	\end{proof}
	
	Despite its simplicity, Lemma~\ref{lem:L2} ensures that for the shadowing property to hold, the initial points of pseudo-orbits must approach $I(\underline{a})$ as $\delta$ tends to zero.
This plays a crucial role in the proof of the following result.
    
	\begin{proposition}\label{prop:X}
		Consider a contractive local IFS $R_{\mathfrak{X}} = (X_j, f_j)$ such that $A_{\mathfrak{X}} \neq \emptyset$. Then, for each $\underline{a} \in \Sigma^+$:
		\begin{enumerate}
			\item If $\underline{a} \notin \Sigma_{\infty}$ then no $(\underline{a}, \delta)$-pseudo-orbit is $(\underline{a}, \varepsilon)$-shadowed by any orbit of $R_{\mathfrak{X}}$;
			\item If $\underline{a} \in \Sigma_{\infty}$, $(y_i)_{i \geq 0}$ is an $(\underline{a}, \delta)$-pseudo-orbit such that \[\varepsilon_0 = \operatorname{dist}_H(\{y_0\}, I(\underline{a})) > 0,\] and $0<\vep<\vep_0$ then it is not $(\underline{a}, \varepsilon)$-shadowed by any orbit of $R_{\mathfrak{X}}$;
			\item If $\underline{a} \in \Sigma_{\infty}$, then for any $\varepsilon > 0$, there exist $N \geq 0$ and $\delta > 0$ such that any $(\sigma^N(\underline{a}), \delta)$-pseudo-orbit (obtained by discarding the first $N$ elements) is $(\sigma^N(\underline{a}), \varepsilon)$-shadowed by an orbit of $R_{\mathfrak{X}}$;
			\item If $\underline{a} \in \Sigma_{\infty}$ and $(y_i)_{i \geq 0}$ is an $(\underline{a}, \delta)$-pseudo-orbit such that $y_0 \in I(\underline{a})$, then it is $(\underline{a}, \varepsilon)$-shadowed by the orbit of $y_0$ provided that $\delta < \varepsilon(1 - \lambda)$.
		\end{enumerate}
	\end{proposition}
	
	\begin{proof}
Item (1) is immediate because, as $\underline a \notin\Sigma_\infty$, no  infinite orbit exists associated to this sequence of symbols. Item (2) is a direct consequence of Lemma~\ref{lem:L2}. 

\smallskip
Let us prove item (3). By Lemma~\ref{lem:L1}, given $\underline a\in \Sigma_\infty$, $x\in I(\underline a)$ and an $(\underline a,\delta)$-pseudo orbit $(x_m)_{m\ge 0}$ one has that  $d(f_{\underline{a}}^m(x), x_m) \leq \frac{1}{1-\lambda} \delta + \lambda^m d(x, x_0)$ for every 
$m\ge 1$
So, given $\varepsilon > 0$, choose $\delta > 0$ such that $\frac{\delta}{1-\lambda} < \varepsilon/2$ and let $N\ge 1$ be such that $\lambda^N \operatorname{diam}(X) < \varepsilon/2$. Then, discarding the first $N$ elements results in a $(\sigma^N(\underline{a}), \delta)$-pseudo-orbit that is $(\sigma^N(\underline{a}), \varepsilon)$-shadowed by the orbit of $f_{\underline{a}}^N(x)$ for any $x \in I(\underline{a})$, as desired. 
\smallskip

Finally, item (4) follows from Lemma~\ref{lem:L1} by setting $x = y_0$, which yields 
$$
d(f_{\underline{a}}^m(y_0), y_m) \leq \frac{\delta}{1-\lambda} < \varepsilon.
$$
This finishes the proof of the proposition.
	\end{proof}
	\bigskip

    \begin{remark}\label{rmk:averageshadowign}
    Item (3) in Proposition~\ref{prop:X} can be reformulated in terms of non-autonomous dynamical systems. In fact, given $\underline a\in \Sigma_\infty$, the non-autonomous dynamical system associated to the sequence of maps $(f_{a_i})_{i\ge 0}$ satisfies an average shadowing property:
for every
$\varepsilon>0$ there exists $\delta>0$ such that for any sequence
$(x_i)_{i\ge 0}$  satisfying
$x_{a_i}\in X_{a_i}$ and 
$d\bigl(f_{a_i}(x_i),x_{i+1}\bigr)<\delta$
there exists $x\in X_{a_0}$ such that
$$
\limsup_{n\to\infty}\frac{1}{n}\sum_{i=0}^{n-1}
d\bigl(f_{a_{i-1}}\circ\cdots\circ f_{a_0}(x),x_i\bigr)<\vep.
$$   
\end{remark}

    \medskip

	\subsection{Proof of Theorem~\ref{thm:Y}}

  The proof of the theorem will follow as a direct consequence of Lemma~\ref{lem:L1} and Proposition~\ref{prop:X}. 

\smallskip
$(1) \Rightarrow (2)$
\smallskip        
        
        Assume that there exists a sequence $\underline{a} \in \Sigma_{\infty}$ such that $\lim_{\delta \to 0} \operatorname{dist}_H(\pi_0(\Gamma_{\underline a, \delta}), I(\underline{a})) \geq \varepsilon_0$. Then for any $0<\varepsilon < \vep_0$ and $\delta>0$ one can find a $(\underline a,\delta)$-pseudo-orbit $(y_i)$ with $d(y_0, I(\underline{a})) \geq \vep$. By Proposition~\ref{prop:X} item (2), this orbit is not $\vep$-shadowed for $0<\varepsilon < \varepsilon_0$, hence the concordant shadowing property fails.

\medskip
$(2) \Rightarrow (1)$
\smallskip        
We wish to prove that  the concordant shadowing property holds.
If $\underline a \notin \Sigma_\infty$ then there exist no infinite $(\underline a, \delta)$-pseudo orbits, and there is nothing to prove. Hence we may assume that $\underline{a} \in \Sigma_{\infty}$.
By assumption, one has that 
		$ 
        \lim_{\delta \to 0} \operatorname{dist}_H(\pi_0(\Gamma_{\underline a, \delta}), I(\underline{a})) = 0.
        $  
        Hence, given $\vep>0$ fixed, one can choose $\delta_0 > 0$ such that $\operatorname{dist}_H(\pi_0(\Gamma_{\underline a, \delta}), I(\underline{a})) < \frac{\varepsilon}{2}$ and $\frac{\delta}{1-\lambda} < \frac{\varepsilon}{2}$ for every $0<\delta < \delta_0$. Thus, given a $(\underline a, \delta)$-pseudo-orbit $(y_i)_{i\ge 0}$ there exists $x \in I(\underline{a})$ with $d(y_0, x) < \varepsilon/2$. By Lemma~\ref{lem:L1}, $d(f_{\underline{a}}^m(x), y_m) < \varepsilon$, hence
$x$ is an $(\underline a,\vep)$-shadow for the pseudo orbit. This finishes the proof of the theorem.
\hfill $\square$

\section{Local IFSs derived from beta-transformations}
\label{sec:local-beta}

In this section we shall prove Theorem~\ref{thm:beta}.
Let $\beta>1$ and set $m=\lceil\beta\rceil$.
Consider the $\beta$-transformation
$ 
T_\beta:[0,1]\to[0,1],\quad \text{given by } T_\beta(x)=\beta x-\lfloor \beta x\rfloor.
$ 
The $\beta$-shift 
$\Sigma_\beta\subset \{0,1,2, \dots, m-1\}^{\mathbb N}$
is defined as
the closure of the set of greedy $\beta$-expansions:
\[
\Sigma_\beta=\overline{\{d_\beta(x):\ x\in[0,1)\}},
\qquad
d_\beta(x)=(d_1(x),d_2(x),\dots),
\quad
d_k(x)=\lfloor \beta T_\beta^{k-1}(x)\rfloor.
\]

We now use that local IFSs are flexible enough so that every piecewise expanding interval map can be dually represented through a contractive local IFS. Let us be more precise in the special case of the $\beta$-transformation.
Consider the $\beta^{-1}$-contractive inverse branches
\[
f^\beta_j(x)=\frac{x+j}{\beta},\qquad 0\le j \le m-1
\]
of $T_\beta$,
set $X^\beta_j=[0,1]$ for every $0\le j \le m-2$ and $X^\beta_{m-1}=\bigl[0,\ \beta-(m-1)\bigr]$.

Consider the contractive local IFS $R_\beta=(X^\beta_j,f^\beta_j)_{0\le j \le m-1}$ and denote its 
code space by $\Sigma^-_\beta$.
Define $\iota:\Sigma^-\to \{1,2, \dots, m\}^{\mathbb N}$ by
$\iota(\underline b)=(b_{-1},b_{-2},b_{-3},\dots).$

\begin{lemma}
The following properties hold:
\begin{enumerate}
    \item For every $\underline b\in\Sigma^-_\beta$, if 
    $x=\pi_\beta(\underline b)$
    then $\iota(\underline b)=d_\beta(x)$;    
    \item $\iota(\Sigma^-_\beta)=\Sigma_\beta.$
\end{enumerate}
\end{lemma}

\begin{proof}
For every $0\le j \le m-1$ and every $y\in[0,1)$ one has
$\beta f_j(y)=y+j$ and $\lfloor y+j\rfloor=j$. Thus, for each $y\in[0,1)$,
\begin{equation}
\label{eq:indices}
T_\beta(f_j(y))=y+j-j=y
\quad\text{and}\quad
\lfloor \beta f_j(y)\rfloor=j.
\end{equation}
Similarly, $f_j(T_\beta(x))=x$ for every $x\in X_j$.

By definition, for each $\underline b\in\Sigma^-_\beta$ 
one has 
$ 
\diam\bigl(V_{[\underline b]_k}\bigr)
\le \beta^{-k}\diam(X_{b_{-k}})
$
for each $k\ge 0$, hence
$ 
\pi_\beta(\underline b)=\bigcap_{k\ge1}V_{[\underline b]_k}
$ 
is a singleton. 
Write $x=\pi_{\mathfrak X}(\underline b)$ and $\underline{c}=\iota(\underline b)\in \Sigma^+$.
As $x\in V_{[\underline b]_k}$, there exists $y_k\in X_{b_{-k}}$ such that
$ 
x=f_{b_{-1}}\circ f_{b_{-2}}\circ\cdots\circ f_{b_{-k}}(y_k).
$ 
Applying \eqref{eq:indices} successively,
\begin{equation}
\label{eq:Tk-r}
T_\beta^{\,k-1}(x)=T_\beta^{\,k-1}\circ  f_{b_{-1}}\circ f_{b_{-2}}\circ\cdots\circ f_{b_{-k}}(y_k) \in X_{b_{-k}}
\end{equation}
and so
$ \Big\lfloor \beta\,f_{b_{-k}}(y_k)\Big\rfloor=b_{-k}$ 
This proves that
$ d_k(x)=b_{-k}=c_k$ 
for each $k\ge 1$. Consequently,
$\iota(\underline b) = d_\beta(x)\in \Sigma_\beta$. This proves item (1) and  also that  $\iota(\Sigma^-_\beta)\subset\Sigma_\beta$.

\smallskip
In order to prove the converse inclusion $\Sigma_\beta \subset \iota(\Sigma^-_\beta)$,
take $x\in[0,1)$ and $\underline c=d_\beta(x)\in\{0,1,2, \dots, m-1\}^{\mathbb N}$, and define $\underline b\in\Sigma^-$ by  $\iota(\underline b)=\underline c$.
We aim at proving that $\underline b\in \Sigma^-_\beta$.
First note that, for each $k\ge 1$, 
$$
c_k=d_k(x)=\lfloor \beta T_\beta^{k-1}(x)\rfloor
= \lfloor T_\beta^{k}(x)\rfloor
$$
which, together
with \eqref{eq:indices},
implies that
\[
x=f_{c_1}\bigl(T_\beta(x)\bigr)
=f_{c_1}\circ f_{c_2}\bigl(T_\beta^2(x)\bigr)
=\cdots
=f_{c_1}\circ\cdots\circ f_{c_k}\bigl(T_\beta^k(x)\bigr).
\]
This ensures that $x\in f_{c_1}\circ\cdots\circ f_{c_k}(X_{c_{-k}})$ or, equivalently, 
$x\in V_{[\underline b]_k}$ (recall $\underline b=\iota(c)$).
Since $k\ge 1$ is arbitrary then 
$$
x\in \bigcap_{k\ge 1} V_{[\underline b]_k} \neq\emptyset 
$$
and so $\underline b\in \Sigma^-_{\beta}$ as desired.
Therefore
$\{d_\beta(x):x\in[0,1)\}\subset \iota(\Sigma^-_\beta).$
Taking the closure, we deduce that
$\Sigma_\beta=\overline{\{d_\beta(x):x\in[0,1)\}}
\subset \iota(\Sigma^-_\beta).$
\end{proof}

\begin{remark}
We recall the following facts: 
(i) the periodic points are dense in coded spaces
(cf. \cite[Section~2]{BP92}), and 
(ii) $\Sigma_\beta$ is a coded shift (see e.g. \cite{CT12}).
Since the code space can be computed as $\Sigma^-_\beta=\iota(\Sigma_\beta)$ then 
the periodic orbits are dense in
$\Sigma^-_\beta$.    
\end{remark}

Recall that $\beta>1$ is a \emph{Parry number} if $d_\beta(1)$ is
eventually periodic, and it is called a \emph{simple Parry number} if
$d_\beta(1)$ is finite. This can be formulated in terms of the orbit of $1$ by the transformation $T_\beta$
as being pre-periodic or for the existence of $k\ge 1$ such that  $T_\beta^k(1)=0$, respectively.
Moreover, it is well known that $\beta>1$ is a Parry number if and only if $\Sigma_\beta$ is sofic (i.e. is a subshift where the language is recognizable by a finite automaton),  and that $\beta>1$ is a simple Parry number if and only if $\Sigma_\beta$ is a subshift of finite type (cf. \cite{Bla89,Sch97} and references therein).
The parameters $\beta>1$ for which
$d_\beta(1)$ is
eventually periodic are at most countable and the orbit of 1 is dense in $[0,1]$ for a Baire generic $\beta$ in $(1,+\infty)$ (cf. \cite{Sch97}). The dependence of the parameter $\beta$ is quite sensitive as non-Parry numbers are also dense in $(1,+\infty)$ (cf. \cite{HS24}).

\begin{proposition}
\label{prop:nonstable-nonparry}
Let $\beta_0>1$ be such that $\beta_0$ is a non-Parry number. Then for every $\varepsilon>0$ there exists $\beta\in(\beta_0-\varepsilon,\beta_0+\varepsilon)$
such that
$\Sigma_{\beta}\neq \Sigma_{\beta_0}.$
In particular, $R_{\beta_0}$ is not combinatorially stable.
\end{proposition}

\begin{proof}
As $\beta_0>1$ is non-Parry then $\Sigma_{\beta_0}$ is not sofic.
Moreover, simple Parry numbers are dense in $(1,+\infty)$ \cite{Pa60}. In particular, 
for any $\vep>0$ there exists 
a simple Parry number $\beta\in({\beta_0}-\varepsilon,{\beta_0}+\varepsilon) \cap (1,+\infty)$.
For such $\beta$, the shift $\Sigma_{\beta}$ is an SFT, hence sofic.
In particular $\Sigma_{\beta}\neq \Sigma_{\beta_0}$ and, ultimately, $\beta\mapsto \Sigma_\beta^-=\iota(\Sigma_\beta)$ is not locally constant
at ${\beta_0}$.
 This proves that $R_{\beta_0}$ is not combinatorially stable.
\end{proof}

%%%%%%%%%%%%
\subsection{Proof of Theorem~\ref{thm:beta}}

We proceed to construct contractive local IFSs with OSC and denseness of periodic orbits which are not combinatorially stable nor satisfy the shadowing property.
Let $T_\beta$ be the $\beta$-transformation with  $4<\beta<5$. 
The space $\Sigma_\beta\subset \{0,1,2, 3,4\}^{\mathbb N}$ can be characterized as
\[
\Sigma_\beta
=
\Bigl\{
\underline c \in\{0,1,2, 3,4\}^{\mathbb N}:\ 
\sigma^k(\underline c)\preceq d_\beta(1)
\ \text{for all } k\ge0
\Bigr\},
\]
where $d_\beta(1)$ is the sequence associated to the $\beta$-expansion of $1$
and $\preceq$ denotes the lexicographical order (see e.g. \cite[Proposition 2.3]{Bla89}).

\smallskip
Consider the contractive local IFS
$R^{(0,2,4)}_{\beta}=(X_i^\beta,f_i^\beta)_{i=0,2,4}$,
which clearly satisfies the open set condition.
Define the compact and shift invariant subset
\begin{equation}
\label{eq:rest-posit}
\Sigma_\beta^{(0,2,4)}
=
\Sigma_\beta\cap\{0,2,4\}^{\mathbb N}.    
\end{equation}
If $\beta$ is a Parry number then
$\Sigma_\beta$ is sofic, hence $\Sigma_\beta^{(0,2,4)}$ is also sofic.
Moreover, since the shift space is discrete the denseness of periodic orbits in $\Sigma_\beta$ it is not hard to check that the same holds for $\Sigma_\beta^{(0,2,4)}$.
 We also observe that 
$$
\Sigma_\beta^{(0,2,4),-}=\iota(\Sigma_\beta^{(0,2,4)})
=\iota(\Sigma_\beta\cap\{0,2,4\}^{\mathbb N})
$$ is the code space for $R^{(0,2,4)}_{\beta}.$
In fact, this is a consequence of the fact that given $\underline b=(\dots,b_{-2},b_{-1})\in\Sigma^-$, the set $V_{[\underline b]_k}$ is non-empty if and only if 
$(b_{-1},b_{-2},\dots,b_{-k})$ appears in some sequence that belongs to $\Sigma_\beta$, and that the local IFS does not allow symbols $\{1,3\}$.    
We need the following key ingredient.

\begin{lemma}\label{le:explicit-beta}
There exists $4<\beta<5$ so that 
$\Sigma^{(0,2,4)}_{\beta}$ is not a subshift of finite type.
\end{lemma}

\begin{proof} 
Let $(n_j)_{j\geq 1}$ be a strictly increasing sequence of positive integers 
and consider the sequence
\[
\underline a:=4\,0^{n_1}\,4\,0^{n_2}\,4\,0^{n_3}\cdots \in \{0,1,2,3,4\}^{\mathbb N}.
\]
Since the sequence is strictly increasing then  $\underline a$ is not pre-periodic and 
\begin{equation}
    \label{eq:construction-a}
    \sigma^k(\underline a)\preceq \underline a \qquad
    \text{ for every $k\ge 1$.}
\end{equation}
By the proof of \cite[Proposition 2.4]{Bla89} one has that $\underline a =d_\beta(1)$ for the (unique) 
$
\beta>1$ defined by
$$
1= 4\beta^{-1} +  4\beta^{-(n_1+2)}+ 4\beta^{-(n_1+n_2+3)}+ 4\beta^{-(n_1+n_2+n_3+4)} + \dots 
$$
or, equivalently, 
$$
\beta= 4 +  4\beta^{-(n_1+1)}+ 4\beta^{-(n_1+n_2+2)}+ 4\beta^{-(n_1+n_2+n_3+3)} + \dots .
$$
This guarantees that $\beta>4$. Moreover, one can choose the sequence $(n_j)_{j\ge 1}$ to be sparse in such a way that 
$$
\beta \le 4 + \sum_{k\ge 1} 4^{-\ell_k} < 5
\quad \text{where $\ell_k=k+\sum_{i=1}^k n_i$.}
$$
By construction $\beta$ is a non-Parry number, hence 
$$
\Sigma_\beta =\Big\{\underline c \in \{0,1,2,3,4\}^{\mathbb N} \colon \sigma^k(\underline c) 
\preceq 
{\underline a} = d_\beta(1) \; \text{for every } k\ge 0\Big\} 
$$ 
is not a subshift of finite type. Furthermore, ~\eqref{eq:construction-a} ensures that $\underline a\in \Sigma_\beta$.
This implies that $\Sigma^{(0,2,4)}_{\beta}$ is not a subshift of finite type.
Indeed, this is a consequence of the fact that the relation $\sigma^k(\underline c)\preceq \underline a$ for every $k\ge 0$ that determines the $\beta$-shift implies that $\underline c=40^{\ell_1} 4 0^{\ell_2}4 0^{\ell_3} \dots \in \Sigma^{(0,2,4)}_{\beta}$ if and only if $\ell_i\ge n_i$ for every $i\ge 1$. This finishes the proof of the lemma.
\end{proof}
\color{black}

Now, we observe that as $\beta$ is a non-Parry number and 
$\sigma\mid_{\Sigma^{(0,2,4)}_{\beta}}$ is not a subshift of finite type, then 
by the classical theorem of Walters \cite{Wa78} the subshift does not have the shadowing property.
Hence the main results ensure that $R_\beta$ is not combinatorially stable and  $A_{\beta}$ does not have the negative shadowing. 
\hfill $\square$

\color{black}

%%%%%%%%%%%%%%%%%%%%%%%%%%%%%%%%%
\section{Topologically stable local IFSs  }\label{sec:contractive-stable}

In Subsection~\ref{sec:top-stab-IFS} we extend the results in \cite{AT25} by proving that every contractive IFS satisfying the open set condition are (strongly) topologically stable, thus proving Theorem~\ref{thm:combinatorial-topological} in this context. In Subsection~\ref{sec:top-stab-local-IFS} finish the proof of Theorem~\ref{thm:combinatorial-topological} by extending this result for local IFSs under an additional combinatorial assumption. 
Finally in Subsection~\ref{sec:charact+topstability}
we prove Theorem~\ref{thm:concordantshadowing-implies-stable} on the characterization of topological stability for local IFSs on compact manifolds.

%%%%%%%%%%%%%%%%%%%%%%%%
\subsection{Contractive IFSs are topologically stable}
\label{sec:top-stab-IFS}

We initiate this subsection by proving that 
pseudo-orbits of contractive IFSs are asymptotic to the attractor. 

\begin{proposition}\label{prop:shadowingIFSG1}
    Let $R=(X,f_j)_{1\le j \le n}$ be a contractive IFS. Then, for any $\vep>0$ there exist $\delta>0$ and $N \in \mathbb{N}$ such that every $(\underline{a},\delta)$-pseudo-orbit $(x_n)_{n\ge 0}$ of $R$  is $(\underline{a},\vep)$-shadowed by a point $x\in X$  
    so that 
    $$d(f_{\underline a}^i(x) , A_R) < \vep/2
    \quad\text{and}\quad
    d(f_{\underline a}^i(x) , x_i) < \vep/2 $$ for every $i \geq N$. 
In particular $d(x_i,A_R) < \vep$ for every $i\ge N$.
\end{proposition}

\begin{proof}
Fix $\vep>0$. Let $\lambda\in (0,1)$ be a contraction constant for $R$ and let $0<\delta<\vep/4$ be given by the concordant shadowing property
so that every $(\underline a,\delta)$-pseudo orbit is  $(\underline a,\vep/4)$-shadowed (cf. \cite{GV06}). Choose $N\ge 1$ so that 
$\lambda^N \diam(X)<\delta$, where $\lambda\in (0,1)$ is an hyperbolicity constant for $R$.
    Take an $(\underline a,\delta)$-pseudo orbit $(x_n)_{n\ge 0}$ of $R$ and denote by $y\in X$ an $(\underline a,\vep/4)$-shadow. Note that 
    $$
    f_{\underline a}^N(y) \in F_R^N(X) 
    $$
    belongs to the $\lambda^N\cdot \diam(X)$-neighborhood of $A_R$. In particular, there exists $z\in A_R$ so that $d(f_{\underline a}^N(y),z)<\delta$. Now, given the $(\underline a,\delta)$ pseudo-orbit 
    $$
    y, f_{a_0}(y), \dots, f_{\underline a}^{N-1}(y), z, f_{a_N}(z), 
    f_{a_{N+1}}(f_{a_N}(z)), \dots 
    $$
there exists $x\in X$ which $(\underline a, \vep/4)$-shadows it. By triangular inequality one concludes that 
     \[
     d(x_k, f_{\sigma^N\underline a}^{k-N}(z))
     \le 
     d(x_k,f_{\underline a}^k(x))+
     d(f_{\underline a}^k(x), f_{\sigma^N\underline a}^{k-N}(z)) < \vep\]
     for every $k\ge N$. This proves the proposition.
\end{proof}

\begin{remark}
The previous result is optimal in the following sense: if $A_R$ is a proper subset of $X$ then any orbit of a point that does not belong to $A_R$ is a $\delta$-pseudo orbit of $R$ (for any  $\delta>0$) but can not be $\vep$-shadowed by an orbit entirely contained in the attractor.   
\end{remark}

\bigskip

\begin{theorem}
\label{prop:shadowingIFSG2}
    Let $R=(X,f_j)_{1\le j \le n}$ be a contractive IFS. If $R$ satisfies 
    the open set condition
    then it is topologically stable in $\mathcal F_C$.
\end{theorem}
\begin{proof}
    Assume that $R=(X,f_j)_{1\le j \le n}$ 
    is a contractive IFS satisfying the open set condition and $\lambda\in (0,1)$ denote its contraction. In particular there exists $0<\vep<L$ so that every contractive IFS $R_{\mathfrak Y}=(X,g_j)_{1\le j \le n}$ which satisfies 
    $D\Big(R_{\mathfrak X}, R_{\mathfrak Y}\Big)<\vep$ satisfies the open set condition because, by triangular inequality,
    $$
    \text{dist}_H(g_i(X),g_j(X))
    \ge \text{dist}_H(f_i(X),f_j(X))
    - 2\vep.
    $$
    Since $R$ is a contractive IFS then it satisfies the concordant shadowing property. In particular, by Corollary~\ref{cor:shadowclose} there exists $0<\delta<\vep$ such that every $(\underline a,\delta)$-pseudo orbit for $S$ is $(\underline a,\vep)$-shadowed by a point for $R$.

\medskip
Fix $\underline a\in  \Sigma^+$. Given $y\in A_{\mathfrak Y}$, one can write
$$y
=\pi_{\mathfrak Y}(\underline b)
= \lim_{k\to+\infty} g_{[\underline b]}^k(z_k)
$$
for each $k\ge 1$, 
where $z_k\in X$ and 
$g_{[\underline b]}^k=g_{b_{-1}} \circ g_{b_{-2}}\circ \dots \circ g_{b_{-k}}$.
Let $k_0\ge 1$ be such that $d(g_{[\underline b]}^k(z_k),y)<\delta$ for every $k\ge k_0$. In consequence, for each $k\ge k_0$ consider the 
$(b_{-k}, \dots , b_{-2}, b_{-1})\ast \underline a,\delta)$ pseudo-orbit 
    $$
    z_k, g_{[\underline b]}(z_k),  \dots, g_{[\underline b]}^{k-1}(z_k), y, g_{a_0}(y), 
    g_{a_{1}}(g_{a_0}(y)), \dots 
    $$
for $R_{\mathfrak Y}$. Using Corollary~\ref{cor:shadowclose}, this pseudo-orbit is 
$((b_{-k}, \dots , b_{-2}, b_{-1})\ast \underline a,\vep)$ shadowed by a point $w_k\in X$ (note that $w_k$ depends on $\underline a$ and $\underline b$, even though the dependence on $\underline b$ will drop due to contraction).
Consider the point
\begin{equation}
    \label{eqdef:comj}
h_{\underline a}(y)=\lim_{k\to+\infty} f_{[\underline b]}^k(w_k) 
= 
\lim_{k\to+\infty} f_{\tilde{\underline a}}^k(w_k) \in X_{a_0}
\end{equation}
where $\tilde{\underline a} = (b_{-k}, \dots, b_{-1})\ast \underline a$.
The map $h_{\underline a}$ is well defined as the previous limit does exist due to uniform contraction, because $\diam (f_{\tilde{\underline a}}^k(X))\le \lambda^k \diam(X)$ for each $k\ge 1$.
Furthermore, by \eqref{eqdef:comj}, one has in addition
\begin{equation}
    \label{eq:codingha}
    h_{\underline a}(y) = \pi_{\mathfrak X}(\underline b) \in A_{\mathfrak X}
\end{equation}
hence $h_{\underline a}: A_{\mathfrak Y} \to A_{\mathfrak X}$ is a well defined map.
By construction, $d(h_{\underline a}(y),y)<\vep$ for every $y\in A_{\mathfrak Y}$, hence $d_{C^0}(h_{\underline a},id_{A_{\mathfrak Y}})<\vep$.

\smallskip
We proceed to show that $h_{\underline a}$ is an homeomorphism. Since $A_{\mathfrak Y}$ is a compact metric space it is enough to show that $h_{\underline a}$ is a continuous bijection. On the one hand, as $S$ satisfies the open set condition one knows that $\pi_{\mathfrak Y}: \Sigma^+ \to A_{\mathfrak Y}$ is an homeomorphism. Thus, given $N\ge 1$ there exists $\delta>0$ so that if 
$\tilde y\in A_{\mathfrak Y}$ with $d(y,\tilde y)<\delta$ then 
$d_{\Sigma^+}(\underline b, \tilde{\underline b})<e^{-N}$ where
$\tilde y=\pi_{\mathfrak Y}(\tilde{\underline b})$. In consequence, 
$b_{-j}=\tilde b_{-j}$ for each $1\le j \le N$ and 
$$
d(h_{\underline a}(y),h_{\underline a}(\tilde y)) < \diam(f_{\tilde{\underline a}}^N(X_N)) \le \lambda^N \diam(X). 
$$
This proves the continuity of $h_{\underline a}.$ The injectivity of $h_{\underline a}$ follows from the injectivity of the maps $\pi_{\mathfrak X}$ and $\pi_{\mathfrak Y}$. Surjectivity is obtained from the fact that each $x\in A_{\mathfrak X}$ can be written as 
$$
x=\pi_{\mathfrak X} (\underline b), \qquad \underline b \in \Sigma^+
$$
and that, by construction, the point $y=\pi_{\mathfrak Y}(\underline b) \in A_{\mathfrak Y}$ satisfies $h_{\underline a}(y) = \pi_{\mathfrak X}(\underline b)=x$.

\smallskip
Finally, it remains to show that 
    \begin{equation*}
        h_{\underline a}\circ g_{a_k}\dots  g_{a_1} g_{a_0}
= f_{a_k}\dots  f_{a_1} f_{a_0} \circ h_{\underline a}
    \end{equation*}
for each $k\ge 1$ (this corresponds to the notion of topological conjugacy with $\tau=id$). Fix an arbitrary $k\ge 1$. Given
$y=\pi(\underline b)\in A_{\mathfrak Y}$ recall that 
$$
y= \lim_{\ell\to \infty} g_{b_{-1}} \circ g_{b_{-2}} \circ \dots g_{b_{-\ell}}(X)
$$
hence
\begin{align*}
g_{a_k}\dots  g_{a_1} g_{a_0}(y) 
 & =
 \lim_{\ell\to \infty} g_{a_k}\circ \dots  \circ g_{a_1} g_{a_0} \circ g_{b_{-1}} \circ g_{b_{-2}} \circ \dots g_{b_{-\ell}}(X)\\
 &
 = \lim_{\ell\to \infty} g^{\ell+k}_{[\underline b \ast (a_0, a_1, \dots, a_k)]}(X) \\
 & =
 \pi_{\mathfrak Y}(\,\underline b \ast (a_0, a_1, \dots, a_k)\,)\in A_{\mathfrak Y}.    
\end{align*}
By construction (recall ~\eqref{eq:codingha}) one deduces that
\begin{align*}
h_{\underline a}(g_{a_k}\dots  g_{a_1} g_{a_0}(y) ) 
& =
\pi_{\mathfrak X}(\,\underline b \ast (a_0, a_1, \dots, a_k)\,) \\
& =  \lim_{\ell\to \infty} f_{a_k}\circ \dots  \circ f_{a_1} f_{a_0} \circ f_{b_{-1}} \circ f_{b_{-2}} \circ \dots f_{b_{-\ell}}(X) \\
& =   f_{a_k}\circ \dots  \circ f_{a_1} f_{a_0} (\lim_{\ell\to \infty} f_{b_{-1}} \circ f_{b_{-2}} \circ \dots f_{b_{-\ell}}(X)) \\
& =   f_{a_k}\circ \dots  \circ f_{a_1} f_{a_0} (\pi_{\mathfrak X}(\underline b))
\\
& =   f_{a_k}\circ \dots  \circ f_{a_1} f_{a_0} (h_{\underline a}(y)).
\end{align*}
Altogether, as $y\in A_{\mathfrak Y}$ and $k\ge 1$ were arbitrary, this proves that $R$ is topologically stable, as desired, and finishes the proof of the proposition.
\end{proof}

\subsection{Contractive local IFSs with concordant shadowing}
\label{sec:top-stab-local-IFS}

We now prove that combinatorial stability together  with the concordant shadowing property implies topological stability for contractive local IFSs. Although the strategy follows the one used in Subsection~\ref{sec:top-stab-IFS}, the local nature of the dynamics requires additional care in the construction of the conjugating map, since admissibility depends on both past and future itineraries.

\smallskip
    Let $(X,d)$ be a compact metric space and assume that $R_{\mathfrak X}=(X_i,f_i)_{1\le i \le n}$ is a combinatorially stable, contractive local IFS. We proceed to prove that $R_{\mathfrak X}$ is topologically stable.
Since $R_{\mathfrak X}$ is combinatorially stable and satisfies the open set condition, there exists an open neighborhood 
$\mathcal V$ of $R_\mathfrak X$ in $\mathcal F_C$ such that every $R_{\mathfrak Y}\in\mathcal V$ is contractive and satisfies
$\Sigma_{\mathfrak X}^-=\Sigma_{\mathfrak Y}^-.$
In particular, the corresponding bilateral shifts
$\widehat\Sigma_{\mathfrak X}$ and $\widehat\Sigma_{\mathfrak Y}$ coincide (cf. Lemma~\ref{le:past-2-future}). 
In particular
\begin{equation}
    \label{eq:structural+}
\widehat\Sigma_{\mathfrak Y}^+=\widehat\Sigma_{\mathfrak X}^+,
\quad \text{for all 
$\mathfrak Y\in\mathcal V$.}
\end{equation}
Up to diminishing $\mathcal V$ if necessary, one can always assume that every contractive local IFS 
$R_{\mathfrak Y}=(Y_i,g_i)\in \mathcal V$ also satisfies the open set condition.

\smallskip
Now, since $R_{\mathfrak Y}$ is contractive and satisfies OSC, 
the construction of the code space (Definition~\ref{def: code map}) ensures that the local code map $\pi_{\mathfrak Y}$ is well defined by the expression
   $$\pi_{\mathfrak Y}(\underline b)= \bigcap_{k \geq 1} V_{[\underline b]_k} = \lim_{k\to \infty} V_{[\underline b]_k}$$
and that defines an H\"older continuous bijection
$\pi_{\mathfrak Y}: \Sigma_{\mathfrak Y}^- \to A_{\mathfrak Y}$.

\smallskip
We are now in a position to complete the proof of Theorem~\ref{thm:combinatorial-topological}. As relation ~\eqref{eq:structural+} ensures that the combinatorics of possible concatenations of maps in $R_{\mathfrak X}$ and $R_{\mathfrak Y}$ coincide, one can mimic the argument in Subsection~\ref{sec:top-stab-IFS}.
By assumption $R_{\mathfrak X}$ satisfies the concordant shadowing property, hence diminishing $\mathcal V$ if necessary, one guarantees that there exists $0<\delta<\vep$ such that every $(\underline a,\delta)$-pseudo orbit for $R_{\mathfrak Y}$ is $(\underline a,\vep)$-shadowed by a point for $R_{\mathfrak X}$ (cf. Corollary~\ref{cor:shadowclose}).

Fix an arbitrary sequence $(\underline b,\underline a)\in  \hat \Sigma_{\mathfrak Y}$. For each $k\ge 1$ there exists $z_k\in Y_{-k}$ in such a way that 
$ 
y
=\pi_{\mathfrak Y}(\underline b)
= \lim_{k\to+\infty} g_{[\underline b]}^k(z_k).
$ 
Choose $k_0\ge 1$ such that $d(g_{[\underline b]}^k(z_k),y)<\delta$ for every $k\ge k_0$. Then the concatenated sequence  
    $$
    z_k, g_{[\underline b]}(z_k),  \dots, g_{[\underline b]}^{k-1}(z_k), y, g_{a_0}(y), 
    g_{a_{1}}(g_{a_0}(y)), \dots 
    $$
forms a $((b_{-k}, \dots , b_{-2}, b_{-1})\ast \underline a,\delta)$ pseudo-orbit 
for $R_{\mathfrak Y}$, hence 
$((b_{-k}, \dots , b_{-2}, b_{-1})\ast \underline a,\vep)$ shadowed by a point $w_k\in X_{b_{-k}}$.
Consider the point
\begin{equation*}
h_{\underline a}(y)=\lim_{k\to+\infty} f_{[\underline b]}^k(w_k) 
= 
\lim_{k\to+\infty} f_{\tilde{\underline a}}^k(w_k) \in X_{a_0}
\end{equation*}
where $\tilde{\underline a} = (b_{-k}, \dots, b_{-1})\ast \underline a$
(the point in $X_{a_0}$ is well defined by uniform contractivity of the local IFS $R_{\mathfrak X}$). 
By combinatorial stability ${\underline b}\in \Sigma_{\mathfrak X}^-$, hence
$h_{\underline a}(y)=\pi_{\mathfrak X}(\underline b) \in A_{\mathfrak X}$.
This defines a continuous 
map $h_{\underline a}: A_{\mathfrak Y} \to A_{\mathfrak X}$ which is $\vep$-$C^0$-close to the identity: if \(y=\pi_Y(b)\) and \(\tilde y=\pi_Y(\tilde b)\) have the same past symbols up to $-N$, then
$d(h_{\underline a}(y),h_{\underline a}(\tilde y))
\le \lambda^N\diam(X)$,
which proves continuity.
\color{black}
The fact that $h_{\underline a}$ is a bijection follows from the same property for both $\pi_{\mathfrak X}$ and $\pi_{\mathfrak Y}$, together with combinatorial stability. In fact, every $x\in A_{\mathfrak X}$ is of the form
\begin{equation}
\label{eq:uniquerepx}
x=\pi_{\mathfrak X} (\underline b), \; \text{for some (unique) } \underline b \in \Sigma^+_{\mathfrak X}    
\end{equation}
and, as $\underline b \in \Sigma^+_{\mathfrak Y}=\Sigma^+_{\mathfrak X}$ the point $y=\pi_{\mathfrak Y}(\underline b) \in A_{\mathfrak Y}$ satisfies $h_{\underline a}(y)=x$. This proves surjectivity while the unique representation in ~\eqref{eq:uniquerepx} implies the injectivity.

\smallskip
Finally, together with Lemma~\ref{le:past-2-future}, combinatorial stability implies that $\hat\Sigma_{\mathfrak Y}=\hat\Sigma_{\mathfrak X}$, and so $(\underline b,\underline a)\in  \hat \Sigma_{\mathfrak X}$, and the point $f_{\underline a}^k(x)$ is well defined for every $k\ge 1$.  The proof of the conjugacy equation
    $h_{\underline a}\circ g_{a_k}\circ\dots \circ g_{a_1} \circ g_{a_0}
= f_{a_k}\circ\dots  \circ f_{a_1} \circ f_{a_0} \circ h_{\underline a}$ for each $k\ge 1$ is identical to the one of Subsection~\ref{sec:top-stab-IFS}.
This proves that $R_{\mathfrak X}$ is topologically stable, and finishes the proof of the theorem. \hfill $\square$

%%%%%%%%%%%%%%%%%%%%%%%%%%%%%%%%%
\subsection{A characterization of topological stability }\label{sec:charact+topstability}

In this section we shall prove Theorem~\ref{thm:concordantshadowing-implies-stable} by showing that, in the framework of local IFSs on compact manifolds of dimension larger or equal than $3$, combinatorial and topological stability in $\mathcal F_0$ suffices to obtain the concordant shadowing property. 

\medskip
Let
  $(X,d)$ be a compact manifold of dimension  $\dim X\ge 3$ and that $X_j\subset X$ are proper sets in the sense that $X_j=\overline{\interior(X_j)}$ for every $1\le j\le n$.
Assume that $R_{\mathfrak X}= \left(X_j, f_j \right)_{1\le j \le n}$ is a combinatorial and topologically stable contractive local IFS. 
By topological stability in $\mathcal F_0$ there exists an open neighborhood $\mathcal V$ of $R_{\mathfrak X}$ in $\mathcal F_0$ so that for every $R_{\mathfrak Y} \in \mathcal V$ 
there exists a topological conjugacy between $R_{\mathfrak X}$ and $R_{\mathfrak Y}$ given by a homeomorphism 
$ \tau:\widehat\Sigma_{\mathfrak X}^+\to \widehat\Sigma_{\mathfrak Y}^+$
  and, for each $\underline a\in\widehat\Sigma_{\mathfrak X}^+$, a homeomorphism 
\[
   h_{\underline a}:A_{\mathfrak X}\to A_{\mathfrak Y}
\]
satisfying the conjugacy relation~\eqref{eq:conjtopconj}.

\medskip
In view of Lemma~\ref{le:finiteshadowequiv}, in order to prove the theorem it is enough to show that 
$R_{\mathfrak X}$  satisfies the finite concordant shadowing property.
Suppose, by contradiction, that this is not the case. Then there exists $\vep_0>0$ such that for every $\delta>0$  there exists 
$\underline a\in \Sigma^+$, $N\ge 1$ and a $(\underline a,\delta)$-pseudo-orbit $(x_k)_{k\ge 0}$ in $A_{\mathfrak X}$ 
so that $(x_k)_{k=0}^N$ is not $(\underline a,\vep_0)$-shadowed by any point in $A_{\mathfrak X}$. 
Consider a monotone decreasing sequence $(\delta_m)_{m\ge 1}$ tending to zero. For each $m\ge 1$ choose 
$\underline a^{(m)}=(a_k^{(m)})_{k\ge 0}\in \Sigma^+$, $N_m\ge 1$ and
a $(\underline a^{(m)},\delta_m)$-pseudo-orbit $(x_k^{(m)})_{k\ge 0}$ in $A_{\mathfrak X}$  so that $(x_k^{(m)})_{k=0}^{ N_m}$ is not $(\underline a^{(m)},\vep_0)$-shadowed.
This means that 
\begin{equation}
\label{eq:PO}
   d\bigl(f_{a_k^{(m)}}(x_k^{(m)}),x_{k+1}^{(m)}\bigr)<\delta_m \qquad \forall k\ge 0 
\end{equation}
while
\begin{equation}
\label{eq:PO2}
\sup_{0\le k\le N_m} d(f_{\underline a^{(m)}}^k(x),x_k) \ge \vep_0, \quad \text{for every } x\in A_{\mathfrak X}.
\end{equation}
By continuity of the maps and inequalities~\eqref{eq:PO} and ~\eqref{eq:PO2} it follows that $N_m$ tends to infinity as $m\to\infty$.
 Since $X$ is a manifold of dimension  $\dim X\ge 3$ and ~\eqref{eq:PO} defines an open condition we may suppose, without loss of generality, that the points in the orbits $(f^k_{\underline a^{(m)}}(x_k^{(m)}))_{1\le k\le N_m}$ do not coincide with any of the points in the pseudo-orbit
$(x_k^{(m)})_{k\ge 1}$ and that no points belong to the boundary of the sets $(X_j)_{1\le j \le n}$.
Hence there exist $C^1$-curves $\ell_{m,k} : [a_{m,k},b_{m,k}] \to X$ parameterized by arc length in such a way that
$$
\ell_{m,k}(a_{m,k}) = f_{a_k^{(m)}}(x_k^{(m)})
\quad
\text{and}
\quad 
\ell_{m,k}(b_{m,k}) = x_{k+1}^{(m)}
$$
for every $0\le k \le N_m-1$ and the images of the curves have length smaller than $\delta_m$ and do not intersect in $X$.

We now proceed to construct, for each $m\ge 1$, a local IFS
$
    R_{\mathfrak X}^{(m)} = (X_j,g^{(m)}_j)_{1\le j\le n}
$ 
which is $\delta_m$-close to $R_{\mathfrak X}$ (in the sense of ~\eqref{def:distanceS}) and for which $(x_k^{(m)})_{k=0}^{N_m}$ becomes an exact orbit, meaning that
\begin{equation}
    \label{eqgoal}
    g_{a_k^{(m)}}(x_k^{(m)}) = x_{k+1}^{(m)}
    \quad\text{for every $0\le k\le N_m$.}
\end{equation}

Now we will perform finitely many small perturbations (with disjoint supports) of the maps $(f_j)_{1\le j \le n}$ in order to obtain continuous maps $(g^{(m)}_j)_{1\le j \le n}$ that, while keeping the same domains,  satisfy $d_{C^0}(f_j, g^{(m)}_j)<\delta_m$ for each $1\le j \le n$ and ~\eqref{eqgoal} holds. 

This construction is recursive. 
Since the map $f_{a^{(m)}_0}$ is continuous and the length of $\ell_{m,0}$ is smaller than $\delta_m$, then one can obtain a 
$\delta_m$-$C^0$-small perturbation of $f_{a_0^{(m)}}$ of the form $h_{m,0 } \circ  f_{a_0^{(m)}}$ where $h_{m,0}$ is a $C^1$-diffeomorphism such that: (i) $d_{C^0}(h_{m,0},id)<\delta_m$, (ii) $S_0:=\{x\in X\colon h_{m,0}(x)\neq x\}$ is a small open neighborhood of $\ell_{m,0}([a_{m,0}, b_{m,0}])$ which does not intersect any of the curves 
$\ell_{m,k}$ with $0< k \le N_m-1$ and 
(iii) 
$ h_{m,0}(f_{a_0^{(m)}}(x_0^{(m)}))= x_{1}^{(m)}.
$ 
Define the local IFS $S^1_{\mathfrak X}=(X_j,f^1_j)_{1\le j \le n}$ where
$$
f^1_j= \begin{cases}
    \begin{array}{cl}
        h_{m,0 } \circ  f_{a_0^{(m)}}  & ,\text{if } j=a_0^{(m)}  \\
      f_j   &  ,\text{otherwise}.
    \end{array}
\end{cases}
$$
By construction, for each $1\le j \le n$ the map $f_j^1$ coincides with $f_j$ on open neighborhoods of the images of the compact disjoint compact curves $\ell_{m,k}$ with $0<k \le N_m$.
Hence, proceeding similarly
there exists a
$\delta_m$-$C^0$-small perturbation of $f_{a_1^{(m)}}$ of the form $h_{m,1 } \circ  f_{a_1^{(m)}}$ where $h_{m,1}$ is a $C^1$-diffeomorphism such that: (i) $d_{C^0}(h_{m,1},id)<\delta_m$, (ii) the support of the perturbation $S_1:=\{x\in X\colon h_{m,1}(x)\neq x\}$ is a small open neighborhood of $\ell_{m,1}([a_{m,1}, b_{m,1}])$ which does not intersect $S_0$ nor any of the curves 
$\ell_{m,k}$ with $2\le k \le N_{m}-1$
and (iii) 
$ h_{m,1}(f_{a_1^{(m)}}(x_1^{(m)}))= x_{2}^{(m)}.
$ 
Define the local IFS $S^2_{\mathfrak X}=(X_j,f^2_j)_{1\le j \le n}$ where
$$
f^2_j= \begin{cases}
    \begin{array}{cl}
        h_{m,1 } \circ  f_{a_1^{(m)}}  & ,\text{if } j=a_1^{(m)}  \\
      f^1_j   &  ,\text{otherwise}.
    \end{array}
\end{cases}
$$
By construction, for each $1\le j \le n$ the map $f_j^2$ coincides with $f_j$ on open neighborhoods of the images of the compact disjoint compact curves $\ell_{m,k}$ with $2 \le k \le N_m-1$.
Proceeding recursively
for $2 \le k \le N_m-1$,
we obtain a local IFS 
$$
S_{\mathfrak X}^{(N_m-1)} = (X_j,f^{N_m-1}_j)_{1\le j\le n}
$$ 
which, relabeling it as
$
    R_{\mathfrak X}^{(m)} = (X_j,g^{(m)}_j)_{1\le j\le n},
$ 
is $\delta_m$-close to $R_{\mathfrak X}$
and satisfies 
\begin{equation}
\label{eq:afterpert}
(g^{(m)}_{\underline a^{(m)}})^k(x_0^{(m)}) = x_k^{(m)}, \quad \text{for every $1\le k \le N_m$}.    
\end{equation}
By the topological stability of $R_{\mathfrak X}$ in $\mathcal F_0$, if $m\ge 1$ is large then there exists a topological conjugacy between $R_{\mathfrak X}$ and $R_{\mathfrak X}^{(m)}$:  for each admissible 
$ 
   \underline a^{(m)}=(a_0^{(m)},a_1^{(m)},\dots)\in \Sigma_{\mathfrak X}^+
$ 
there is a homeomorphism $h_{\underline a}:A_{\mathfrak X}\to A_{\mathfrak Y}$ such that
\[
   h_{\underline a}\circ
   f_{\underline a^{(m)}}^k
   \;=\;
   (g_{\underline a^{(m)}}^{(m)})^k \circ h_{\underline a}
   \quad\text{for all }k\ge 0.
\]
The latter, together ~\eqref{eq:afterpert}, implies that  
\[
  f_{a_k^{(m)}}\circ\cdots\circ f_{a_0^{(m)}}(y)
  = x_{k+1}^{(m)}
  \quad\text{for some }y\in A_{\mathfrak X},
\]
hence contradicting ~\eqref{eq:PO2}. This proves that $R_{\mathfrak X}$ satisfies the concordant shadowing property and finishes the proof of the theorem.
\hfill  $\square$

\section{Examples and applications}
\label{sec:examples}

The purpose of this section is to complement the abstract theory developed above with concrete examples and applications that illustrate both the strength and the limitations of our main results. We consider explicit classes of local iterated function systems arising from symbolic models and geometric constructions, for which the coding dynamics can be analyzed in a precise way. These examples serve not only to demonstrate the applicability of the general framework, and exhibit phenomena such as combinatorial instability and failure of shadowing, emphasizing the sharpness of the assumptions in the main results.

\subsection{Contractive local IFSs and fractal geometry}

Our first example shows that the hyperbolicity of the local IFS is not enough to guarantee the concordant shadowing property.

\begin{example} 
    \label{ex:shift2} 
    ({\bf A symbolic contractive local IFS with no concordant shadowing}) 
 Consider the metric space $X = \{0,1,2\}^{\mathbb{N}}$ endowed with the usual distance $d((x_n),(y_n))=2^{-N}$, where $N\ge 1$ is the least integer $n\ge 1$ so that $x_n\neq y_n$ (and $N=+\infty$ if such an integer does not exist). Define the following maps and their corresponding domains:
\[
\begin{aligned}
f_0(x) &:= (0, x_1, x_2, \ldots), & X_0 &:= \{0,1\}^{\mathbb{N}}, \\
f_1(x) &:= (1, x_1, x_2, \ldots), & X_1 &:= \{0,1\}^{\mathbb{N}}, \\
f_2(x) &:= (0,0, x_1+x_2, x_2+x_3, \ldots), & X_2 &:= X_0 \cup X_1 = \{0,1\}^{\mathbb{N}}.
\end{aligned}
\]
It is simple to check that $\operatorname{Lip}(f_j) \le 1/2$  for every $ 0\le j \le 2$, hence $R_{\mathfrak X} = (X_j, f_j)_{0\le j \le 2}$ is a contractive local IFS. 
Moreover, it is known that 
$A_{\mathfrak X}=\{0,1\}^{\mathbb{N}} \cup Q$ where $Q := f_2(\{0,1\}^{\mathbb{N}})$
(see \cite{OV25a}).

We claim that $R_{\mathfrak X}$ does not satisfy the concordant shadowing property. Given $\delta>0$ arbitrary let $N=N(\delta)=\lfloor -\ln \delta \rfloor +1$. Then, any two sequences in $X$ that coincide in the first $N$ coordinates are within distance $\delta$. Consider the $\delta$-pseudo orbit in $X$ given by
\[
\begin{aligned}
& \underline x^0 := (0,1, 0, 0, \dots,\boxed{0}, 0,0, \ldots) \\
& \underline x^1  :=f_2(\underline x^0)= (0,0,1,1, 0, \dots,\boxed{0}, 0,0, \ldots) \\
& \underline  x^i :=f^i_0(\underline x^1)\qquad {\forall 1\le i \le N-3} \\
& \underline  x^{N-2} :=f_0(\underline x^{N-3})
= (0,0,0,0, 0, \dots,\boxed{0}, 1,1, 0, 0, \ldots) \\
& \underline  x^{N-1} 
= (0,0,0,0, 0, \dots,\boxed{0}, 0,0, 0, 0, \ldots) \\
& \underline  x^{n} 
= f_2^{n-N+1}(0,0,0,0, 0, \dots,\boxed{0}, 0,0, 0, 0, \ldots) \qquad {\forall n\ge  N}
\end{aligned}
\]
where the box marks the $N$th position of the sequence (note that $d(f_0(\underline x^{N-2}), \underline x^{N-1})<\delta$ while all other elements in the sequence are obtained by iteration). Such a pseudo-orbit is not shadowable by a true orbit. Indeed, defining  
$\underline a= (2,0, 0, \dots, 0, \boxed{2},2,2,\dots )\in \Sigma^+$, if there would exist $\vep>0$ and 
$\underline x\in X$ so that $d(f_{a_n}\circ f_{a_{n-1}}\circ \dots \circ f_{a_2}\circ f_{a_1}(\underline x),\underline x^k)<\vep$. The only point that can be iterated by $f_2\circ f_2$ is $\underline 0$ which is not close to $\underline x^0$.
This proves the claim. 

It follows from \cite{OV25a} that the code space $\Sigma_{\mathfrak X}^-$ is not a subshift of finite type. In particular, by Corollary~\ref{thm:negativeshadowing-classification}, the local IFS $R_{\mathfrak X}$ does not satisfy the negative shadowing property.
\end{example}

We recall that Theorem~\ref{thm:usc} ensures that local attractors are upper-semicontinuous as a function of the domains of local IFSs.
However, continuity on the Vietoris topology seems unlikely in general. A more natural question concerns the regularity of the Hausdorff dimension of the local attractors for special parameterized families of local IFSs.

  \begin{example}\label{ex: Barnsley Superfractals}  
Fix $0\le \beta \le 1$. Consider $X \subset [0,1]^2$ the equilateral triangle (see Figure~\ref{fig:sierp_barns} top-left) and the local IFS $R_{\mathfrak{X}}=(X_i, f_i)_{1\le i \le 3}$ where
	\[\begin{cases}
		f_1(x,y)= 1/2(x,y)\\
		f_2(x,y)= 1/2(x+0,y+1)\\
		f_3(x,y)= 1/2(x+1/2,y+\sqrt{3}/2)
	\end{cases}
	\]
	and the restrictions are $\mathfrak{X}=(X_j)_{1\le j \le 3}$ given by 
	$X_1=X_2=X$ and $X_3=X^\beta_3:=\{(x,y) \in X\;\colon \; y \le \beta\}$.
For $\beta=1$ we recover the classical Sierpinski triangle, whose code-space is $\{1,2,3\}^{\mathbb{N}}$, while for $\beta=0$ the local attractor is the line segment $[0,1] \times \{0\}$. In general, the local attractor $A^\beta_{\mathfrak X}$ will be a strictly smaller set (see Figure~\ref{fig:sierp_barns} for a numerical simulation when $\beta=0.3$).  This happens because different orbits may disappear (to become endpoints) in arbitrary times of the iteration.  Indeed, given a point $(x,y)$ if we iterate many times by $f_3$ we will reach a point $(x',y')$ such that $y'> \beta$ so we cannot apply $f_3$. However such number of iterates depends on how close the point is from the line segment $[0,1] \times\{0\}$. On the other hand, intercalating $f_3$ with long periods of only $f_1$ and $f_2$ we can always apply $f_3$, since these two maps will necessarily decrease the $y$ coordinate under $\beta$.    \color{black} 
    \begin{figure}[H]
        \centering
        \includegraphics[width=0.30\linewidth]{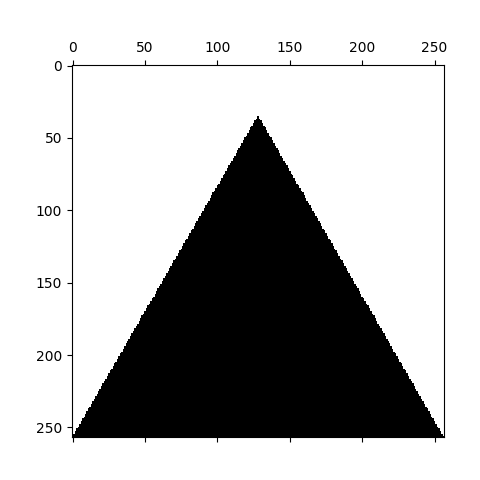}
        \includegraphics[width=0.30\linewidth]{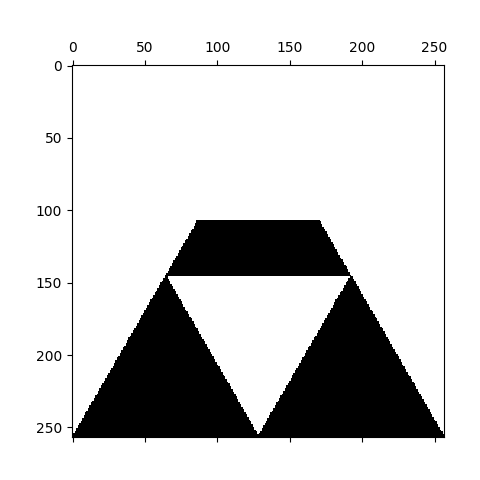}\\
        \includegraphics[width=0.30\linewidth]{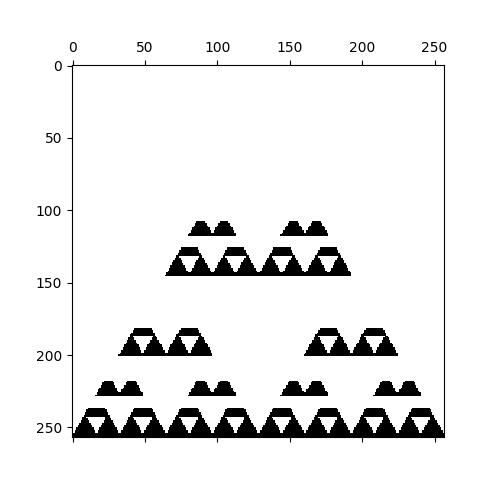}
        \includegraphics[width=0.30\linewidth]{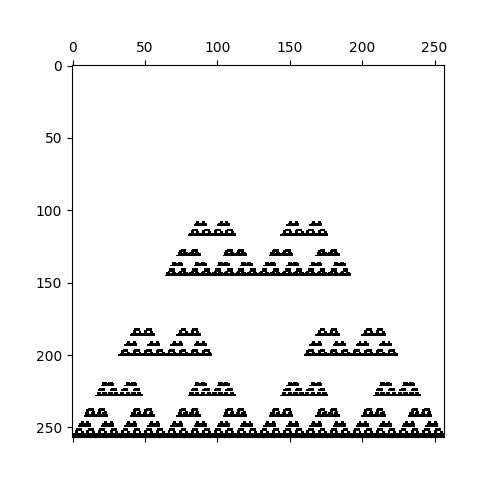}
        \caption{Iterations $F_{\mathfrak X}^k(X)$ for $k=0,1,4,9$ and $\beta=0.3$.}
        \label{fig:sierp_barns}
    \end{figure}

\end{example}

Although Theorem~\ref{thm:usc} ensures upper semicontinuity of the local attractor, the following geometric example highlights the subtle bifurcation mechanism through which lower semicontinuity may fail. More precisely, we construct an increasing family of restrictions $\mathfrak{X}^t$ whose local attractor $A_{\mathfrak{X}}^{t}$ undergoes transitions near $t=\tfrac{1}{4}$ and $t=\tfrac{3}{4}$, providing a clear geometric description of the phenomenon.
\color{black}

\begin{example}\label{examp: non-semicont} 
	Consider $X=[0,1]^2$ and the IFS $R=(X, f_1, f_2,f_3,f_4)$ where
	\[\begin{cases}
		f_1(x,y)= 1/4(x+1,y+1)\\
		f_2(x,y)= 1/4(x+0,y+1)\\
		f_3(x,y)= 1/4(x+1,y+0)\\
		f_4(x,y)= (1-c)(x,y) + c (3/4,3/4),  3/4<c<1, c=0.8
	\end{cases}
	\]
	and the parameterized family  $\mathfrak{X}^t=(X_1^t,X_2^t,X_3^t,X_4^t)$, for
	$$X_1^t=X_2^t=X_3^t={[0, 0.5]}
    \quad
    \text{and} \quad
    X_4^t=[1-t,1]\times[1-t,1]
    $$
    for each $t \in [1/8,1].$ 
	A simple inspection shows that if $t<s$ then $X_4^t \subset X_4^s$ (the first ones are constant), and for $|t-s|<\delta$ one has $\text{dist}_H^{\max} (\mathfrak{X}^t,\mathfrak{X}^s) < \sqrt{2}{\delta}$.
	It is not hard to check  that 
    \[A_{\mathfrak X}^{t}=
    \begin{cases}
		\text{a Sierpinski gasket for } 1/8 \leq t <1/4 \text{ (see Figure~\ref{fig:sierp_usc} top-left)}\\
		\text{the union of a Sierpinski gasket with the point $(3/4,3/4)$ for } 1/4 \leq t <3/4 \\\text{(see Figure~\ref{fig:sierp_usc} top-right)}\\
	      \text{the union of a Sierpinski gasket with a fractal set close to the point $(3/4,3/4)$ } \\ \text{for } 3/4 \leq t <1 \text{ (see Figure~\ref{fig:sierp_usc} bottom-left)}\\
          \text{the union of a Sierpinski gasket with renormalized copy of the Sierpinski gasket} \\ \text{close to the point $(3/4,3/4)$ for $t=1$}
 \text{(see Figure~\ref{fig:sierp_usc} bottom-right)}\\
	\end{cases}
	\]
    \begin{figure}[htb]
        \centering
    \includegraphics[width=0.42\linewidth]{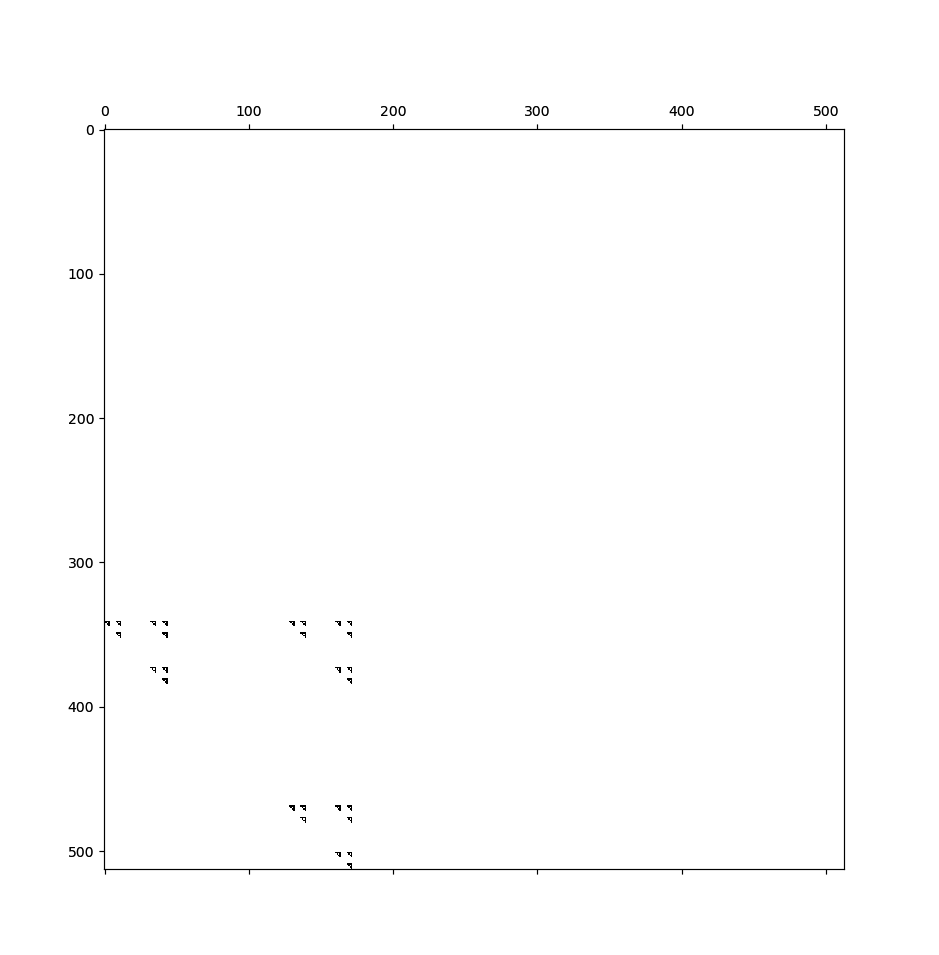}
        \includegraphics[width=0.40\linewidth]{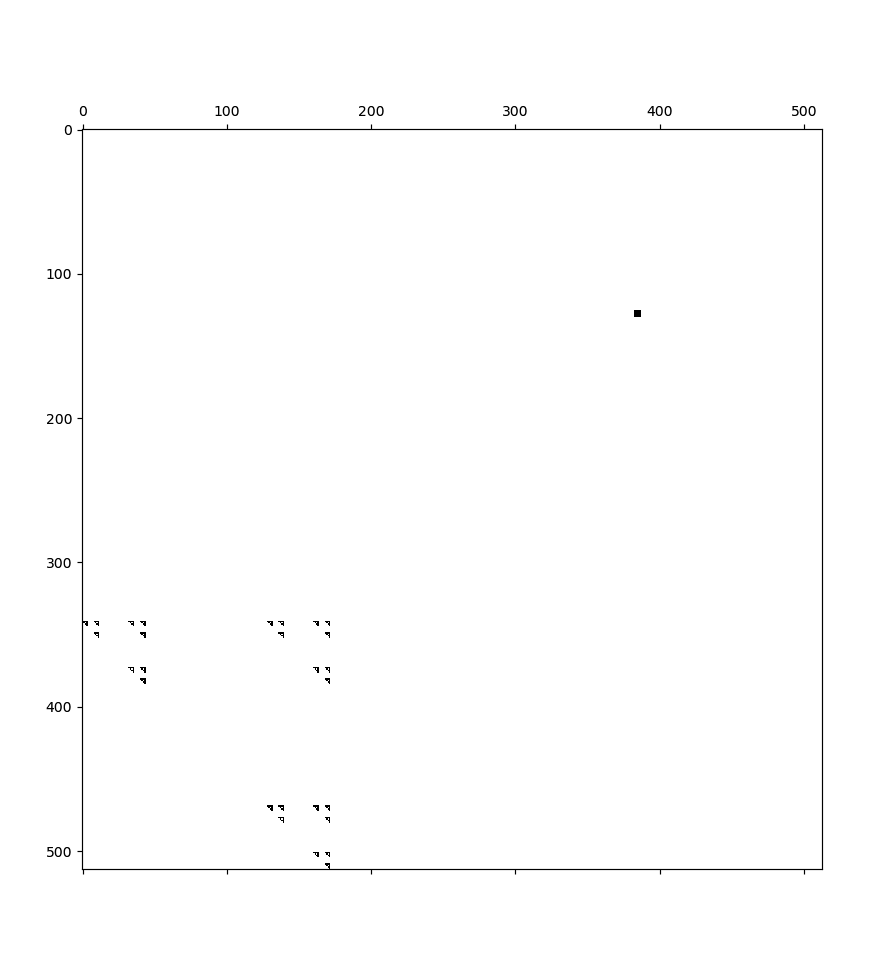}\\
\includegraphics[width=0.40\linewidth]{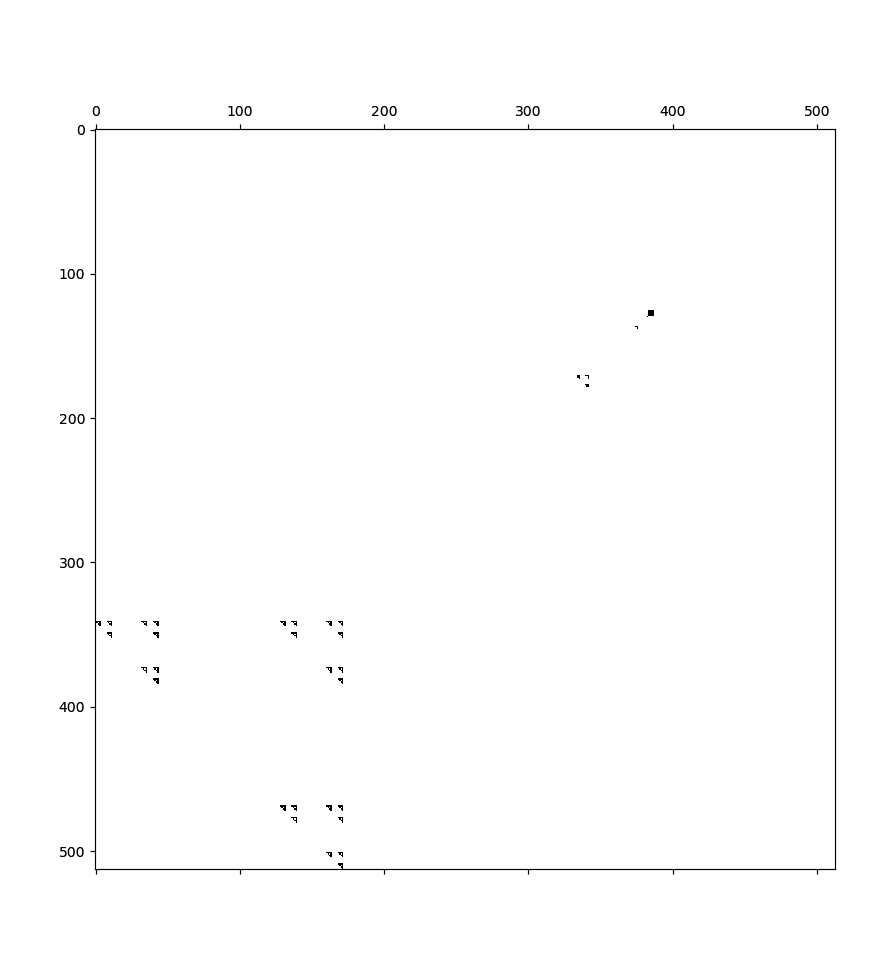}     \includegraphics[width=0.40\linewidth]{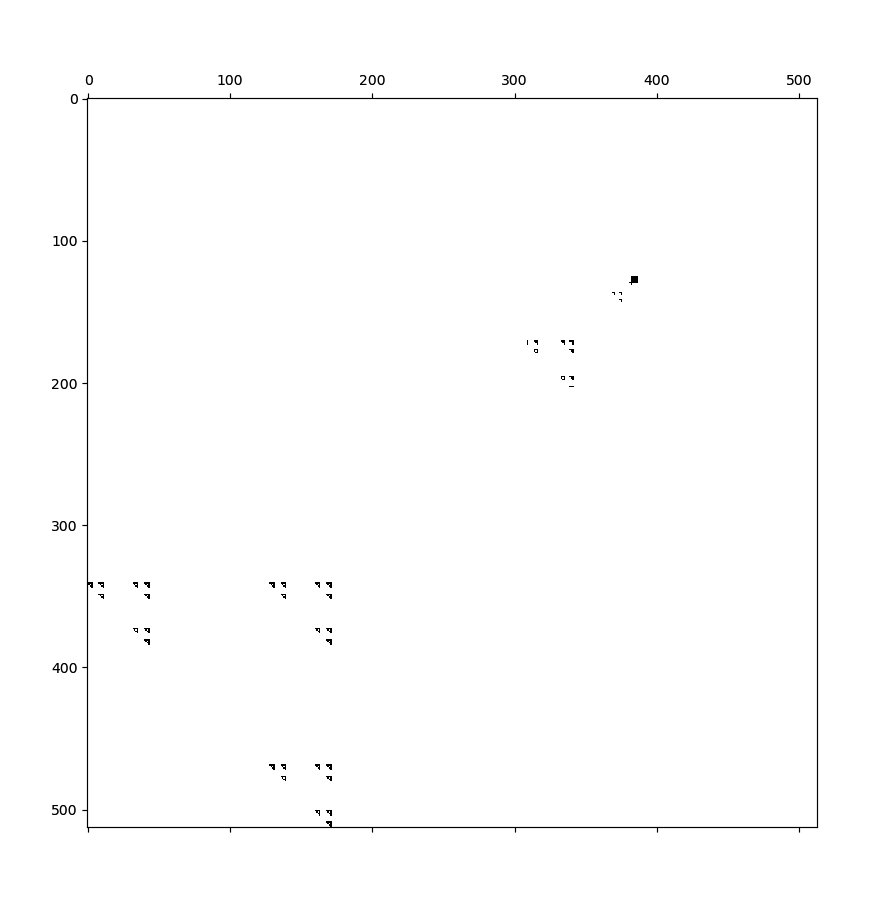}
        \caption{Family of local attractors $A_{\mathfrak X}^{t}$.  The last one is a sequence of scaled copies of the original Sierpinski gasket.}
        \label{fig:sierp_usc}
    \end{figure}
    The attractor in this example does not vary lower semi-continuously (see the jump of one point at $t=1/4$ or or in $t=3/4$). 
\end{example}

\subsection{Pseudogroup actions}\label{subsec:pseudo}

We conclude this section with applications of our results to the framework of group, semigroup and pseudogroup actions.
It has been proved that expansive actions of finitely generated (semi)groups satisfying the shadowing property are topologically stable \cite{CCP25,CL18}. 
Let us recall some concepts.

\begin{definition}
    \label{def:shadow-groups}
Let $G=\langle G_1\rangle$ be a finitely generated (semi)group acting continuously on a compact metric space $(X,d)$ by $\phi: G\times X \to X$. Given $\delta>0$, a family $(x_g)_{g\in G}\subset X$ is called a $\delta$--pseudo-orbit of the action if
\[
d\bigl(\phi_h(x_g),x_{gh}\bigr)<\delta
\quad\text{for all } g\in G \text{ and } h\in G_1.
\]
The action $\phi$ satisfies the \emph{shadowing property} if for every $\varepsilon>0$ there exists $\delta>0$ such that for any $\delta$--pseudo-orbit $(x_g)_{g\in G}$ there exists $x\in X$ satisfying
\[
d\bigl(\phi_g(x),x_g\bigr)<\varepsilon
\quad\text{for all } g\in G.
\]
\end{definition}

The previous notion requires pseudo-orbits to remain close by the action of a set of finite generators for the group, and turns out to be quite rigid. More precisely, few group actions are known to have the shadowing property, namely: (i) 
equicontinuous actions of a finitely generated group on the Cantor set (ii) subshifts of finite type in $\{1,2,\dots,n\}^G$ where $G$ is a finitely generated group (cf. \cite{CL18}).
Furthermore, results by Osipov and Tikhomirov \cite{OT16}, show that the shadowing property for group actions is highly sensitive to the algebraic structure of the acting group, and not merely to the dynamics of individual group elements: there exist $\mathbb{Z}^2$-actions for which each generator has the shadowing property, while the action itself fails to have shadowing.

\smallskip In view of this rigidity, weaker notions of shadowing have been introduced in \cite{Xu24}, inspired by non-autonomous dynamical systems and by the concept of orbital specification developed in \cite{RV16}. Let us recall the concept of orbital shadowing for free semigroup actions.

\begin{definition}
\label{eq:defMa}
Let $G_1=\{g_1, g_2,\dots, g_m\}$ be a collection of continuous maps on a compact metric space $X$ and let $G = \langle G_1 \rangle$ be the free semigroup generated by $G_1$.   
A sequence $(x_n)_{n \geq 0} \subset X$ is called a \emph{$\delta$-pseudo-orbit} if for every $n \geq 0$ there exists $a_n\in\{1,2, \dots, m\}$  such that
\begin{equation}\label{eq:pseudofree}
d\bigl(g_{a_n}(x_n), x_{n+1}\bigr) < \delta
\quad \text{for every } n \geq 0.
\end{equation}
The action of $G$ on $X$ has the (orbital) \emph{shadowing property} if for every $\varepsilon > 0$ there exists $\delta > 0$ such that for any $\delta$-pseudo-orbit $(x_n)_{n \geq 0}$ of the form ~\eqref{eq:pseudofree} there exist a point $x \in X$ satisfying
\[
d\bigl(g_{a_{n+1}} \circ \cdots \circ g_{a_1} \circ g_{a_0}(x), x_n\bigr) < \varepsilon
\quad \text{for all } n \geq 0.
\]
\end{definition}

The notion of orbital shadowing for free semigroup actions corresponds precisely to the concept of concordant shadowing for IFSs. Indeed, a pseudo-orbit of the form \eqref{eq:pseudofree} is exactly an $(\underline a,\delta)$-pseudo orbit for the IFS $R=(X,g_i)_{1\le i \le m}$, with $\underline a=(a_{n})_{n\ge 0}$.
Such a parallel can be formulated in the more general setting of finitely generated pseudogroup actions, which we now introduce.

\begin{definition}
\label{def:pseudogroup}
Let $(X,d)$ be a compact metric space.
A \emph{pseudogroup} $G$ acting on $X$ is a collection of homeomorphisms
$ 
g:U_g \longrightarrow V_g$ on
open sets $U_g,V_g\subset X$
such that:
\begin{enumerate}
    \item $\mathrm{id}_X\in G$ (or, equivalently, $\bigcup \{ U_g\colon g\in G\}=X$)
    \item if $g\in G$, then $g^{-1}\in G$;
    \item if $g,f \in G$ then $g \circ f : f^{-1}(V_f \cap U_g)\to g(V_f \cap U_g)$ belongs to $G$;
    \item if $g\in G$ and $U\subset U_g$ is open, then
    $g|_U\in G$.
\end{enumerate}
\end{definition}

\begin{definition}
Let $G$ be a finitely generated free pseudogroup acting on a compact metric space $(X,d)$, generated by $G_1=\{g_1,g_2,\dots,g_m\}$
where each
$g_i:U_i\to V_i$
is a homeomorphism between open subsets of $X$. A sequence $(x_n)_{n\geq 0}\subset X$ is called a \emph{$\delta$-pseudo-orbit} of the pseudogroup action if for every $n\geq 0$ there exists
$a_n\in\{1,2,\dots,m\}$
such that
\[
x_n\in U_{a_n}
\quad\text{and}\quad
d\bigl(g_{a_n}(x_n),x_{n+1}\bigr)<\delta.
\]
We say that the pseudogroup action has the \emph{orbital shadowing property} if for every $\varepsilon>0$ there exists $\delta>0$ such that every $\delta$-pseudo-orbit
$(x_n)_{n\geq 0}$
is $\varepsilon$-shadowed by a true orbit.
\end{definition}

To the best of our knowledge, the corresponding problem of relating the shadowing property and stability for pseudogroup actions has not yet been addressed.
Associated to a finitely generated pseudogroup action as above, one can consider the local IFS 
\begin{equation}
    \label{eq:IFSpseudo}
    R_{G}=(\overline{U_{g_i}},g_i)_{1\le i \le m}.
\end{equation}
In this way, the orbital shadowing property of the pseudogroup action is naturally related to the concordant shadowing property of the corresponding local IFS. As a direct consequence of Theorems~\ref{thm:shadowing-usc}--\ref{thm:concordantshadowing-implies-stable}, one obtains the following immediate consequence:

\begin{maincorollary}
\label{cor:J}
Let 
$G=\langle G_1\rangle$ be a finitely generated free pseudogroup acting continuously on a compact metric space $(X,d)$. Assume that
(i) each generator $g_i$ is contractive for every $1\leq i\leq m$, (ii) the closure of the sets $(g_i(U_{g_i}))_{1\le i \le m}$ are pairwise disjoint, and (iii) the pseudogroup action has the robust orbital shadowing property.
Then there exists an open neighborhood $\mathcal V$ of the local IFS $R_G$ given by ~\eqref{eq:IFSpseudo}  such that,
for every pseudogroup action
$\widetilde G=\langle \widetilde g_1,\dots,\widetilde g_m\rangle$
such that $R_{\widetilde G} \in\mathcal V$
the following holds:
\begin{enumerate}
    \item the composition $g_{a_n}\circ \dots \circ g_{a_1}\circ g_{a_0}$ is admissible if and only if  the same occurs for $\tilde g_{a_n}\circ \dots \circ \tilde g_{a_1}\circ \tilde g_{a_0}$;
        \item if the generators are contractive, then every admissible pseudo-orbit
    admits a unique shadowing point;
        \item the corresponding local attractors are topologically conjugate, namely,
    the associated local IFSs
    $R_{\widetilde G}$ and $ R_G$
    are topologically stable.
\end{enumerate}
\end{maincorollary}

\begin{remark}
The previous corollary shows that robust orbital shadowing guarantees the persistence of the admissible symbolic structure and the topological stability of the pseudogroup action under small perturbations of both the maps and domains. This is particularly relevant in the setting of good pseudogroups, where the domains are slightly reduced so that their closures become compact (we refer the reader to \cite{Wal04} for the definition of good pseudogroups).      
\end{remark}

\bigskip
\subsection*{Acknowledgments} This work was initiated during a one-year visit of ERO to University 
of Aveiro. 
ERO was partially supported by MATH-AMSUD under the project GSA/CAPES, Grant 88881.694479/2022-01, and by CNPq Grant 408180/2023-4.
PV was partially supported by CIDMA under the Portuguese Foundation for Science and Technology  (FCT, https://ror.org/00snfqn58)  Multi-Annual Financing Program for R\&D Units
 https://doi.org/10.54499/UID/04106/2025, and by Agence Nationale de la Recherche (ANR) under the ANR 2020 funding programme, project THERMOGAMAS.

%%%%%%%%%%%%%%%%%%%%%%%%%%%%%%%%

%%%%%%%%%%%%%%%%%%%%%%%%%%%%

\end{document}